\documentclass[12pt]{article}
\usepackage{amssymb}
\usepackage{amsmath}
\usepackage{amsthm}
\usepackage{color}
\usepackage{comment}
\usepackage{enumitem}
\usepackage{graphicx}
\begin{document}

\def\cF{\mathcal{F}}
\def\cL{\mathcal{L}}
\def\cO{\mathcal{O}}
\def\cP{\mathcal{P}}
\def\cS{\mathcal{S}}
\def\cX{\mathcal{X}}
\def\cY{\mathcal{Y}}
\def\cZ{\mathcal{Z}}

\definecolor{myBlue}{rgb}{0,.4,.6}

\newcommand{\myStep}[1]{\noindent {\bf Step #1}.}

\newcommand{\removableFootnote}[1]{}

\newtheorem{theorem}{Theorem}[section]
\newtheorem{lemma}[theorem]{Lemma}
\newtheorem{proposition}[theorem]{Proposition}

\theoremstyle{definition}
\newtheorem{definition}{Definition}[section]

\title{
Subsumed homoclinic connections and infinitely many coexisting attractors in piecewise-linear maps.
}
\author{
D.J.W.~Simpson and C.P.~Tuffley\\
Institute of Fundamental Sciences\\
Massey University\\
Palmerston North\\
New Zealand}
\maketitle

\begin{abstract}

We establish an equivalence between infinitely many asymptotically stable periodic solutions and subsumed homoclinic connections for $N$-dimensional piecewise-linear continuous maps. These features arise as a codimension-three phenomenon. The periodic solutions are single-round: they each involve one excursion away from a central saddle-type periodic solution. The homoclinic connection is subsumed in the sense that one branch of the unstable manifold of the saddle solution is contained entirely within its stable manifold. The results are proved by using exact expressions for the periodic solutions and components of the stable and unstable manifolds which are available because the maps are piecewise-linear. We also describe a practical approach for finding this phenomenon in the parameter space of a map and illustrate the results with the three-dimensional border-collision normal form.

\end{abstract}

% MSC codes:
% 	37G25 -- Bifurcations connected with nontransversal intersection
%		37G15 -- Bifurcations of limit cycles and periodic orbits

%=====================================================================
\section{Introduction}
\label{sec:intro}
\setcounter{equation}{0}

One of the most important global features of a discrete-time dynamical system (or map) is that of a homoclinic connection.
Homoclinic connections are given by intersections between the stable and unstable manifolds of an invariant set and are central to our understanding of chaos.
If the intersections are transverse, then they are not broken by small perturbations to the map.
For this reason, transverse homoclinic connections are robust structures.
Their existence implies the presence of a Smale horseshoe, infinitely many unstable periodic orbits, and chaotic dynamics \cite{PaTa93}.

As the parameters of a smooth map are varied, transverse homoclinic connections can be created (or destroyed)
in homoclinic tangencies.
Here, stable and unstable manifolds intersect tangentially.
A homoclinic tangency is a codimension-one phenomenon,
thus to understand the nature of the dynamics associated with the tangency as the system is perturbed,
it suffices to vary a single parameter.

In the simplest setting: that the invariant set is a saddle fixed point of a two-dimensional map,
these dynamics were first studied rigorously in \cite{GaSi72,GaSi73}.
The simplest nearby bifurcations are those of single-round periodic solutions
(involving one excursion away from the fixed point).
These bifurcations are described by a one-dimensional quadratic map through the process of renormalisation.
The single-round periodic solutions can be asymptotically stable, 
but sufficiently close to a non-degenerate homoclinic tangency,
asymptotically stable single-round periodic solutions do not coexist.
Other invariant sets exist near homoclinic tangencies, such as multi-round periodic solutions,
and indeed the complete bifurcation structure is fractal \cite{GoBa02}.
Attractors may coexist, and infinitely many attractors coexist on parameter sets known as Newhouse regions \cite{Ne74,Ro83}.

In more than two dimensions, only the slowest stable and unstable directions about the saddle fixed point are important to the dynamics.
This is because almost all forward orbits that approach the fixed point do so tangent to the slowest stable direction,
and almost all backward orbits that approach the fixed point asymptotically do so tangent to the slowest unstable direction.
Thus if the eigenvalues associated with these directions are real and of multiplicity one,
then as in the two-dimensional case the renormalised dynamics are described by a one-dimensional quadratic map.
Cases involving complex eigenvalues lead to different maps,
and it has been shown that four maps suffice to describe all possible scenarios \cite{GoSh96,GoSh97,GoTu05}.

Various types of degenerate homoclinic tangencies have been studied.
Cubic and other higher-order tangencies yield a similar basic bifurcation structure \cite{Da91,GoTu07}.
A degeneracy in the global reinjection mechanism allows for the existence
of infinitely many elliptic single-round periodic solutions in area-preserving maps \cite{GoSh05,GoGo09}.
For a codimension-two scenario at which the branches of the stable and unstable manifolds that intersect are coincident,
an unfolding reveals single-round periodic solutions existing between pairs of saddle-node bifurcations \cite{HiLa95}.
A codimension-three homoclinic tangency corresponding to a Shilnikov-Hopf bifurcation is unfolded in \cite{ChRo99}.
Also, for piecewise-smooth maps, homoclinic connections can be created in homoclinic corners \cite{Si16b}.

This paper concerns piecewise-linear maps that are continuous but non-differentiable on a
codimension-one manifold termed the switching manifold.
The global dynamics of such maps describe the local dynamics of border-collision bifurcations \cite{DiBu08,Si16}.
Such dynamics can involve multiple attractors \cite{KaMa98,DuNu99,ZhMo08d}.
This paper builds on a result of \cite{Si14} for such maps in two dimensions.
There it was shown that, under certain conditions,
if the map has infinitely many stable single-round periodic solutions,
then it must have a particular coincident homoclinic connection that is codimension-three.
As the map is perturbed from this codimension-three scenario,
the number of coexisting attractors scales with $\frac{\ln(\varepsilon)}{\ln(\lambda)}$,
where $\varepsilon$ is the size of the parameter change and $\lambda$ is the associated stable eigenvalue \cite{Si14b}.
Examples with the same underlying mechanism had been given earlier
in an area-preserving scenario \cite{DoLa08}
and for a three-component piecewise-smooth map \cite{GaTr83}.

In this paper the result of \cite{Si14} is extended to $N$-dimensions (Theorem \ref{th:homoclinic}).
Here the unstable manifold is again one-dimensional, but the stable manifold is now $(N-1)$-dimensional.
Branches of the stable and unstable manifolds therefore cannot be coincident.
Instead we find that one branch of the unstable manifold is
contained within the stable manifold --- we say it is {\em subsumed}.
We also provide a converse result (Theorem \ref{th:infinity}) showing that, under certain conditions,
the presence of a subsumed homoclinic connection,
as characterised using an orbit with two points on the switching manifold,
implies the existence of infinitely many asymptotically stable single-round periodic solutions.

The remainder of this paper is organised as follows.
In \S\ref{sec:symbolic} we introduce a general piecewise-linear continuous map $f$
and explain how orbits and periodic solutions of $f$ are encoded symbolically.
We also characterise the stability and admissibility of periodic solutions
and provide a simple result regarding line segments.

In \S\ref{sec:symHC} we provide rationale for the assumptions
placed upon the symbolic itineraries in Theorems \ref{th:infinity} and \ref{th:homoclinic}.
This is achieved by investigating the consequences that an orbit $\{ y_i \}$,
homoclinic to an $\cX$-cycle (periodic solution with itinerary $\cX$) is of a simple type.
%by considering the simplest type of homoclinic orbit of $f$.
Then in \S\ref{sec:subsumed} we show that if the dimension of the unstable manifold of the $\cX$-cycle is one
and $\{ y_i \}$ has two particular points on the switching manifold of $f$,
then, under certain conditions, the $\cX$-cycle has a subsumed homoclinic connection, Proposition \ref{pr:subsumed}.
In \S\ref{sec:thms} we then state and discuss Theorems \ref{th:infinity} and \ref{th:homoclinic}.

In \S\ref{sec:examples} we explain how the codimension-three scenario can be identified numerically
in a map with at least three parameters and provide three specific examples using
the three-dimensional border-collision normal form.
Theorems \ref{th:infinity} and \ref{th:homoclinic} are then proved in \S\ref{sec:proofs}
using a slew of algebraic, analytical, and geometric arguments.
The over-riding strategy is to work in a coordinate system centred about one point
of the $\cX$-cycle and with axes that at least partially align with the stable and unstable subspaces of the $\cX$-cycle.
Finally \S\ref{sec:conc} provides a summary and outlook for future studies.

%=====================================================================
\section{The basic properties of a general piecewise-linear map}
\label{sec:symbolic}
\setcounter{equation}{0}

Let $f : \mathbb{R}^N \to \mathbb{R}^N$ be defined by
\begin{equation}
f(x) = \begin{cases}
A_L x + b \;, & e_1^{\sf T} x \le 0 \;, \\
A_R x + b \;, & e_1^{\sf T} x \ge 0 \;,
\end{cases}
\label{eq:f}
\end{equation}
where $A_L$ and $A_R$ are real-valued $N \times N$ matrices and $b \in \mathbb{R}^N$.
Throughout this paper, $e_i$ denotes the $i^{\rm th}$ coordinate vector of $\mathbb{R}^N$.
Thus, in particular, $e_1^{\sf T} x$ is the first component of $x \in \mathbb{R}^N$.

On the hyperplane $e_1^{\sf T} x = 0$, which we denote by $\Sigma$,
the function $f$ is continuous but not differentiable (unless $A_L = A_R$).
The assumption of continuity of $f$ on $\Sigma$ implies that $A_L$ and $A_R$
differ in only their first columns.
That is, there exists $\xi \in \mathbb{R}^N$ such that
\begin{equation}
A_R = A_L + \xi e_1^{\sf T} \;.
\label{eq:xi}
\end{equation}

In the context of dynamical systems,
$f$ is a piecewise-linear continuous map with switching manifold $\Sigma$.
It is general in the sense that any piecewise-linear continuous function with two components
can be put in the form \eqref{eq:f} by choosing coordinates so that
the boundary between the components is $\Sigma$.
Given any $x_0 \in \mathbb{R}^N$,
we let $\{ x_i \}$ denote the forward orbit of $x_0$ under $f$,
as defined by $x_{i+1} = f(x_i)$, for all $i \ge 0$.

%-------------------------------------------------------------------------------
\subsection{Using words to encode finite parts of orbits}
\label{sub:words}

It is helpful to encode the itinerary of orbits of $f$, relative to $\Sigma$, using the alphabet $\{ L,R \}$.
Throughout this paper we write a word $\cX$ of length $n$ as
\begin{equation}
\cX = \cX_0 \cdots \cX_{n-1} \;,
%\label{eq:X}
\nonumber
\end{equation}
where $\cX_i \in \{ L,R \}$ for each $i = 0,\ldots,n-1$.

%~~~~~~~~~~~~~~~~~~~~~~~~~~~~~~~~~~~~~~~~~~~~~~~~~~~~~~~~~~~~~~~~~~~~~~~~~~~~~~~
\begin{definition}
Given $x_0 \in \mathbb{R}^N$ and a word $\cX$ of length $n$,
we say that $x_0$ {\em follows} $\cX$ under $f$ if
\begin{equation}
\begin{gathered}
\cX_i = L {\rm ~whenever~} e_1^{\sf T} x_i < 0 \;, \\
\cX_i = R {\rm ~whenever~} e_1^{\sf T} x_i > 0 \;,
\end{gathered}
\label{eq:admCondition}
\end{equation}
over all $i = 0,\ldots,n-1$.
\label{df:follow}
\end{definition}

If $x_i \in \Sigma$, for some $i = 0,\ldots,n-1$,
then there is no restriction on $\cX_i$ in Definition \ref{df:follow}.
In this case $x_0$ follows more than one word of length $n$.
For example, the orbit shown in Fig.~\ref{fig:orbitEx} has $x_0, x_3 \in \Sigma$
and so $x_0$ follows four words of length five under $f$:
\begin{equation}
LLLLR \;, \qquad
LLLRR \;, \qquad
RLLLR \;, \qquad
RLLRR \;.
\label{eq:wordsEx}
\end{equation}
Here we formalise this observation.
Throughout this paper $\cX^{\overline{i}}$ denotes the word formed
from $\cX$ by flipping the symbol $\cX_i$ (from $L$ to $R$, or vice-versa).

%%%%%%%%%%%%%%%%%%%%%%%%%%%%%%%%%%%%%%%%%%%%%%%%%%%%%%%%%%%%%
\begin{figure}[b!]
\begin{center}
\setlength{\unitlength}{1cm}
\begin{picture}(8,4)
\put(0,0){\includegraphics[height=4cm]{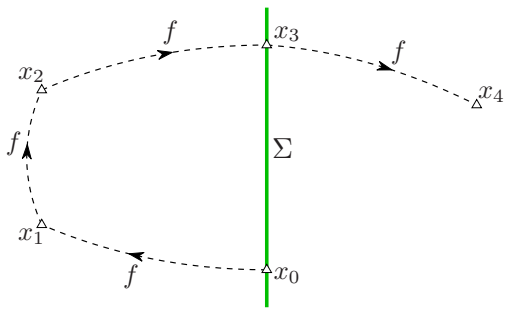}}
\put(4.08,.39){\footnotesize $x_0$}
\put(.68,.9){\footnotesize $x_1$}
\put(.68,3.05){\footnotesize $x_2$}
\put(4.08,3.6){\footnotesize $x_3$}
\put(6.8,2.8){\footnotesize $x_4$}
\put(2.08,.34){\footnotesize $f$}
\put(.52,2.1){\footnotesize $f$}
\put(2.6,3.58){\footnotesize $f$}
\put(5.64,3.32){\footnotesize $f$}
\put(4.06,2){\small $\Sigma$}
\end{picture}
\caption{
A sketch of the initial part of the forward orbit of a point $x_0 \in \mathbb{R}^N$
that follows each of the words in \eqref{eq:wordsEx} under the map $f$.
\label{fig:orbitEx}
}
\end{center}
\end{figure}
%%%%%%%%%%%%%%%%%%%%%%%%%%%%%%%%%%%%%%%%%%%%%%%%%%%%%%%%%%%%%

%...............................................................................
\begin{lemma}
Suppose $x_0$ follows $\cX$ under $f$ and $x_i \in \Sigma$ for some $i = 0,\ldots,n-1$.
Then $x_0$ also follows $\cX^{\overline{i}}$ under $f$.
\label{le:follow}
\end{lemma}

%-------------------------------------------------------------------------------
\subsection{Line segments}
\label{sub:lineSegments}

We use $\cL$ to denote line segments in $\mathbb{R}^N$,
with square brackets when endpoints are included in the segment and round brackets otherwise.
For instance, given $y,z \in \mathbb{R}^N$,
\begin{equation}
\cL[y,z] = \left\{ (1-s) y + s z ~\middle|~ 0 \le s \le 1 \right\}.
%\label{eq:cL}
\nonumber
\end{equation}

If $y$ and $z$ lie on the same side of $\Sigma$, then each $x \in \cL[y,z]$ also lies on this side of $\Sigma$.
In this case, the image of $\cL[y,z]$ under $f$ is given using only one component of $f$.
Both components of $f$ are affine, hence this image is another line segment.
By repeating this observation and extending it to allow points on $\Sigma$ we obtain the following result\removableFootnote{
Here is a proof:

For ease of explanation and without loss of generality, suppose $\cX_0 = R$.
Since $y_0$ and $z_0$ both follow $R$ under $f$,
we have $e_1^{\sf T} y_0 \ge 0$ and $e_1^{\sf T} z_0 \ge 0$.
Thus for any $x_0 \in \cL[y_0,z_0]$, we also have $e_1^{\sf T} x_0 \ge 0$,
and so $x_0$ also follows $R$ under $f$.
Moreover $x_1 \in \cL[y_1,z_1]$ because $f_R(x) = A_R x + b$ is affine.
By repeating this argument $(n-1)$ times,
we conclude that $x_0$ follows $\cX$ under $f$, as required.
}.

%...............................................................................
\begin{lemma}
Suppose $y_0$ and $z_0$ follow $\cX$ under $f$.
Then every $x_0 \in \cL[y_0,z_0]$ also follows $\cX$ under $f$.
\label{le:lineSegment}
\end{lemma}

%-------------------------------------------------------------------------------
\subsection{Periodic solutions}
\label{sub:perSolns}

Let
\begin{equation}
f_L(x) = A_L x + b \;, \qquad
f_R(x) = A_R x + b \;,
%\label{eq:fLfR}
\nonumber
\end{equation}
denote the two components of $f$.

%~~~~~~~~~~~~~~~~~~~~~~~~~~~~~~~~~~~~~~~~~~~~~~~~~~~~~~~~~~~~~~~~~~~~~~~~~~~~~~~
\begin{definition}
An {\em $\cX$-cycle} is an $n$-tuple, $\left\{ x^{\cX}_i \right\}$, for which
\begin{equation}
f_{\cX_0} \!\left( x^{\cX}_0 \right) = x^{\cX}_1 ,~ 
f_{\cX_1} \!\left( x^{\cX}_1 \right) = x^{\cX}_2 ,\;\ldots,~
f_{\cX_{n-1}} \!\left( x^{\cX}_{n-1} \right) = x^{\cX}_0 \;.
\label{eq:Xcycle}
\end{equation}
\label{df:Xcycle}
\end{definition}

If $x^{\cX}_0$ follows $\cX$ under $f$,
then the $\cX$-cycle is in fact an orbit
of $f$ and said to be {\em admissible}.
The next result is an immediate consequence of Lemma \ref{le:follow}.

%...............................................................................
\begin{lemma}
Suppose $\left\{ x^{\cX}_i \right\}$ is an $\cX$-cycle
and $x^{\cX}_j \in \Sigma$ for some $j = 0,\ldots,n-1$.
Then $\left\{ x^{\cX}_i \right\}$ is also an $\cX^{\overline{j}}$-cycle.
\label{le:XcycleFlip}
\end{lemma}

For example, the periodic solution shown in Fig.~\ref{fig:XcycleEx}
has two points on $\Sigma$ and is an $\cX$-cycle where $\cX$ is any of the four words
\eqref{eq:wordsEx}.

%%%%%%%%%%%%%%%%%%%%%%%%%%%%%%%%%%%%%%%%%%%%%%%%%%%%%%%%%%%%%
\begin{figure}[b!]
\begin{center}
\setlength{\unitlength}{1cm}
\begin{picture}(8,4)
\put(0,0){\includegraphics[height=4cm]{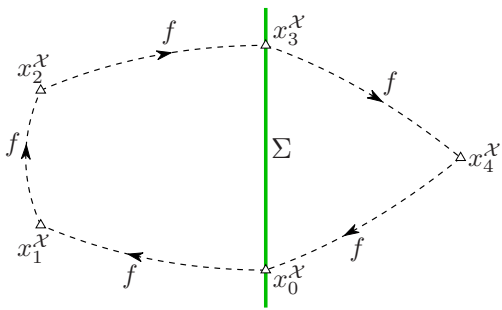}}
\put(4.08,.28){\footnotesize $x^{\cX}_0$}
\put(.68,.69){\footnotesize $x^{\cX}_1$}
\put(.68,3.09){\footnotesize $x^{\cX}_2$}
\put(4.08,3.6){\footnotesize $x^{\cX}_3$}
\put(6.68,1.9){\footnotesize $x^{\cX}_4$}
\put(2.08,.34){\footnotesize $f$}
\put(.52,2.1){\footnotesize $f$}
\put(2.6,3.58){\footnotesize $f$}
\put(5.54,2.9){\footnotesize $f$}
\put(5.1,.7){\footnotesize $f$}
\put(4.06,2){\small $\Sigma$}
\end{picture}
\caption{
A periodic solution of $f$ that is an $\cX$-cycle for any of the words in \eqref{eq:wordsEx}.
\label{fig:XcycleEx}
}
\end{center}
\end{figure}
%%%%%%%%%%%%%%%%%%%%%%%%%%%%%%%%%%%%%%%%%%%%%%%%%%%%%%%%%%%%%

%-------------------------------------------------------------------------------
\subsection{Existence, uniqueness, and stability of periodic solutions}
\label{sub:existenceStability}

Let
\begin{equation}
f_{\cX} = f_{\cX_{n-1}} \circ \cdots \circ f_{\cX_0}
%\label{eq:fX}
\nonumber
\end{equation}
represent the composition of $f_L$ and $f_R$ in the order specified by $\cX$.
The point $x^{\cX}_0$ of an $\cX$-cycle is a fixed point of $f_{\cX}$.
Furthermore, each $x^{\cX}_i$ is a fixed point of $f_{\cX^{(i)}}$,
where throughout this paper we use $\cX^{(i)}$ to denote the $i^{\rm th}$ left cyclic permutation of $\cX$,
that is
\begin{equation}
\cX^{(i)} = \cX_i \cdots \cX_{n-1} \cX_0 \cdots \cX_{i-1} \;.
\nonumber
\end{equation}
The function $f_{\cX}$ is affine, and
\begin{equation}
f_{\cX}(x) = M_{\cX} x + P_{\cX} b \;,
%\label{eq:fX2}
\nonumber
\end{equation}
where
\begin{align}
M_{\cX} &= A_{\cX_{n-1}} \cdots A_{\cX_0} \;, \label{eq:MX} \\
P_{\cX} &= I + A_{\cX_{n-1}} + A_{\cX_{n-1}} A_{\cX_{n-2}} + \ldots +
A_{\cX_{n-1}} \cdots A_{\cX_1} \;. \label{eq:PX}
\end{align}
If $1$ is not an eigenvalue of $M_{\cX}$, then the $\cX$-cycle is unique and
\begin{equation}
x^{\cX}_i = \left( I - M_{\cX^{(i)}} \right)^{-1} P_{\cX^{(i)}} b \;,
\quad {\rm for~all~} i = 0,\ldots,n-1 \;.
%\label{eq:xXi}
\nonumber
\end{equation}
If the $\cX$-cycle is admissible with no points on $\Sigma$,
then every point in some open region containing $x^{\cX}_0$ follows $\cX$ under $f$.
The $\cX$-cycle is therefore a stable periodic solution of $f$
if and only if all eigenvalues of $M_{\cX}$ have modulus less than or equal to $1$.
The $\cX$-cycle is furthermore an asymptotically stable periodic solution of $f$
if and only if all eigenvalues of $M_{\cX}$ have modulus strictly less than $1$.
If instead the $\cX$-cycle involves points on $\Sigma$,
then its stability is much more difficult to characterise \cite{DoKi08}.

%-------------------------------------------------------------------------------
\subsection{Using symbol sequences to encode orbits}
\label{sub:orbits}

We encode orbits of $f$ with bi-infinite sequences
\begin{equation}
\cS = \cdots \cS_{-1} \cS_0 \cS_1 \cdots \;,
\nonumber
\end{equation}
where $\cS_i \in \{ L,R \}$ for all $i \in \mathbb{Z}$.

%~~~~~~~~~~~~~~~~~~~~~~~~~~~~~~~~~~~~~~~~~~~~~~~~~~~~~~~~~~~~~~~~~~~~~~~~~~~~~~~
\begin{definition}
An {\em $\cS$-orbit} is a sequence $\left\{ x^{\cS}_i \right\}$ for which
\begin{equation}
x^{\cS}_{i+1} = f_{\cS_i} \!\left( x^{\cS}_i \right) \;,
\nonumber
\end{equation}
for all $i \in \mathbb{Z}$.
\label{df:Sorbit}
\end{definition}

Admissibility of $\cS$-orbits parallels that of $\cX$-cycles.
Specifically, if an $\cS$-orbit satisfies \eqref{eq:admCondition} for all $i \in \mathbb{Z}$
(using $\cS$ in place of $\cX$),
then it is an orbit of $f$ and we say it is admissible.

%=====================================================================
\section{Symbolic representations for simple homoclinic connections}
\label{sec:symHC}
\setcounter{equation}{0}

In this section we convert elementary observations regarding homoclinic connections
into formal results stated in terms of the words and sequences introduced in \S\ref{sec:symbolic}.
These results motivate the assumptions contained in Theorems \ref{th:infinity} and \ref{th:homoclinic}
and establish some restrictions on the words that can be used in these theorems.

Let $\cX$ be a word of length $n$,
let $\left\{ x^{\cX}_i \right\}$ be an admissible $\cX$-cycle\removableFootnote{
I see no problems with the arguments in this section in the case
that the $\cX$-cycle has one or more points on $\Sigma$.
},
and let $\{ y_i \}$ be an orbit homoclinic to $\left\{ x^{\cX}_i \right\}$.
That is, the sequence $\{ y_i \}$ converges to the $\cX$-cycle as $i \to \pm \infty$.
Fig.~\ref{fig:HCEx} shows an example.

%%%%%%%%%%%%%%%%%%%%%%%%%%%%%%%%%%%%%%%%%%%%%%%%%%%%%%%%%%%%%
\begin{figure}[b!]
\begin{center}
\setlength{\unitlength}{1cm}
\begin{picture}(8,8)
\put(0,0){\includegraphics[height=8cm]{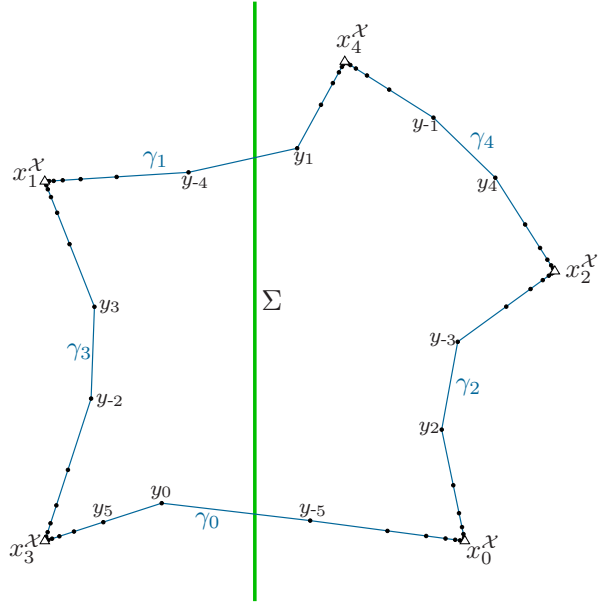}}
\put(6.4,.6){\footnotesize $x^{\cX}_0$}
\put(.37,5.62){\footnotesize $x^{\cX}_1$}
\put(7.72,4.34){\footnotesize $x^{\cX}_2$}
\put(.32,.62){\footnotesize $x^{\cX}_3$}
\put(4.67,7.4){\footnotesize $x^{\cX}_4$}
\put(3.68,3.9){\small $\Sigma$}
\put(2.8,1.05){\footnotesize \color{myBlue} $\gamma_0$}
\put(2.1,5.85){\footnotesize \color{myBlue} $\gamma_1$}
\put(6.26,2.83){\footnotesize \color{myBlue} $\gamma_2$}
\put(1.09,3.3){\footnotesize \color{myBlue} $\gamma_3$}
\put(6.45,6.1){\footnotesize \color{myBlue} $\gamma_4$}
\put(4.18,1.2){\scriptsize $y_{\hspace{.8mm}5}$}
\put(4.33,1.178){\tiny -}
\put(2.61,5.56){\scriptsize $y_{\hspace{.8mm}4}$}
\put(2.76,5.538){\tiny -}
\put(5.91,3.47){\scriptsize $y_{\hspace{.8mm}3}$}
\put(6.06,3.448){\tiny -}
\put(1.48,2.66){\scriptsize $y_{\hspace{.8mm}2}$}
\put(1.63,2.638){\tiny -}
\put(5.68,6.32){\scriptsize $y_{\hspace{.8mm}1}$}
\put(5.83,6.298){\tiny -}
\put(2.2,1.44){\scriptsize $y_0$}
\put(4.1,5.88){\scriptsize $y_1$}
\put(5.77,2.29){\scriptsize $y_2$}
\put(1.52,3.91){\scriptsize $y_3$}
\put(6.54,5.52){\scriptsize $y_4$}
\put(1.43,1.19){\scriptsize $y_5$}
%\put(0,0){\scriptsize $y_{\hspace{.8mm}1}$}
%\put(.15,-.022){\tiny -}
\end{picture}
\caption{
A sketch of the phase space of $f$
showing an orbit $\{ y_i \}$ homoclinic to an $\cX$-cycle, where $\cX = RLRLR$.
The orbit is an $\cS$-orbit, where $\cS$ is given by \eqref{eq:S} and $\cY = LR$.
The piecewise-linear paths $\gamma_0,\ldots,\gamma_4$, \eqref{eq:gammai}, are also shown.
Notice, $\gamma_0$ is a path from $x^{\cX}_0$ to $x^{\cX}_3$ because $d = 3$, as given by \eqref{eq:d}
with $n = 5$ and $p = 2$ (the lengths of $\cX$ and $\cY$).
\label{fig:HCEx}
}
\end{center}
\end{figure}
%%%%%%%%%%%%%%%%%%%%%%%%%%%%%%%%%%%%%%%%%%%%%%%%%%%%%%%%%%%%%

The orbit $\{ y_i \}$ is an $\cS$-orbit, where
\begin{equation}
\cS = \cX^{\infty} \cY \cX^{\infty} \;,
\label{eq:S}
\end{equation}
for some word $\cY$ of length $p$.
For simplicity we assume that $y_0$ corresponds to the symbol $\cY_0$.
Equation \eqref{eq:S} is short-hand for the precise description
\begin{equation}
\cS_i = \begin{cases}
\cX_{i {\rm \,mod\,} n} \;, & i < 0 \;, \\
\cY_i \;, & i = 0,\ldots,p-1 \;, \\
\cX_{(i-p) {\rm \,mod\,} n} \;, & i \ge p \;,
\end{cases}
\label{eq:S2}
\end{equation}
where, for any $i \in \mathbb{Z}$,
we write $i {\rm ~mod~} n$ to denote the integer in $\{ 0,\ldots,n-1 \}$
that differs from $i$ by an integer multiple of $n$.
We also assume $\cX_0 \ne \cY_0$
so that as we follow the points of the $\cS$-orbit in order,
the word $\cY$ starts at the first point at which periodicity following $\cX$ ceases\removableFootnote{
Originally I thought not to include this assumption as so to be more general,
however this requires more notation (two different $\hat{\imath}$'s and $\hat{\jmath}$'s),
plus in the statement of Theorem \ref{th:infinity} we include this assumption
and it is convenient to have the results here, particularly Lemma \ref{le:XConstraint},
apply specifically to this case.
}.

We now look at the behaviour of the $\cS$-orbit more closely.
For each $i = 0,\ldots,n-1$, let
%we join consecutive $n^{\rm th}$ iterates of the $\cS$-orbit, including $y_i$, to form
\begin{equation}
\gamma_i = \bigcup_{j \in \mathbb{Z}} \cL \!\left[ y_{(j-1) n + i}, y_{jn+i} \right),
\label{eq:gammai}
%\nonumber
\end{equation}
as shown in Fig.~\ref{fig:HCEx}.
Let us first discuss $\gamma_0$.
By \eqref{eq:S2}, as $j \to -\infty$ the sequence $\{ y_{j n} \}$ converges to $x^{\cX}_0$.
Similarly, as $j \to \infty$, the sequence $\{ y_{j n} \}$ converges to $x^{\cX}_d$, where
\begin{equation}
d = -p {\rm ~mod~} n \;.
\label{eq:d}
\end{equation}
Therefore $\gamma_0$ is a path from $x^{\cX}_0$ to $x^{\cX}_d$.
Similarly each $\gamma_i$ is a path from $x^{\cX}_i$ to $x^{\cX}_{(d + i) {\rm \,mod\,} n}$.

Let
\begin{equation}
\cP = \bigcup_{i=0}^{n-1} \gamma_i \cup x^{\cX}_i
%\label{eq:closedPath}
\nonumber
\end{equation}
denote the union of the paths and the $\cX$-cycle.
In general, $\cP$ is a collection of loops involving the points of the $\cX$-cycle.
If
\begin{equation}
d \ne 0 \;, \qquad
\gcd(d,n) = 1 \;,
\label{eq:dConstraints}
\end{equation}
where $\gcd$ abbreviates greatest common divisor,
then $\cP$ consists of a single loop, as in Fig.~\ref{fig:HCEx}.
%is a single closed path connecting the points of the $\cX$-cycle.
%Other values of $d$ give rise to structures that are outside the scope of this paper.

Next we look for where $\cP$ intersects $\Sigma$.
Since $\cX_0 \ne \cY_0$, we have $\cS_{-n} \ne \cS_0$. %by \eqref{eq:S2}.
Assuming for simplicity that $y_{-n}$ and $y_0$ do not lie on $\Sigma$,
they therefore lie on different sides of $\Sigma$.
Thus the line segment $\cL \!\left[ y_{-n}, y_0 \right)$ of $\cP$ intersects $\Sigma$.
Since $\cP$ is a collection of loops, it must intersect $\Sigma$ elsewhere.
We now consider the simplest scenario: that $\cP$ intersects $\Sigma$ at only one other point.

More precisely, we suppose that the symbols $\cS_{(j-1)n+i}$ and $\cS_{j n + i}$,
that correspond to the endpoints of $\left[ y_{(j-1) n + i}, y_{jn+i} \right]$,
differ only for $i = j = 0$ and one other pair of values, call them $\hat{\imath}$ and $\hat{\jmath}$.
That is,
\begin{equation}
\cS_{(j-1)n+i} \ne \cS_{j n + i} {\rm ~if~and~only~if~}
i = j = 0 {\rm ~or~} i = \hat{\imath}, j = \hat{\jmath} \;.
\label{eq:twoCrossingsAlt}
\end{equation}
In view of \eqref{eq:S2}, we must have $\hat{\jmath} \ge 0$.
Assuming $\hat{\imath} \ne 0$
(which is the case whenever $\cX$ contains both $L$'s and $R$'s)
the elements of $\cS$ are given in terms of the elements of $\cX$ by
\begin{equation}
\begin{split}
\cS_{j n} &= \cX_0 {\rm ~if~and~only~if~} j < 0 \;, \\
\cS_{j n + \hat{\imath}} &= \cX_{\hat{\imath}} {\rm ~if~and~only~if~} j < \hat{\jmath} \;, \\
\cS_{j n + i} &= \cX_i {\rm ~otherwise} \;.
\end{split}
%\label{eq:twoCrossings}
\end{equation}
%In summary, by assuming that the sequences of $n^{\rm th}$ iterates of the $\cS$-orbit
%connect all points of the $\cX$-cycle while only crossing $\Sigma$ twice,
%we are led to the formula \eqref{eq:twoCrossings}.

%...............................................................................
\begin{lemma}
Let $\cX$ and $\cY$ be words of length $n$ and $p$ with $\cX_0 \ne \cY_0$.
Suppose $\cS$, as given by \eqref{eq:S}, satisfies \eqref{eq:twoCrossingsAlt}.
Then
\begin{equation}
\cX \cY = \left( \cY \cX \right)^{\overline{0} \, \overline{\alpha}} \;,
\label{eq:XY}
\end{equation}
where $\alpha = \hat{\jmath} n + \hat{\imath}$.
\label{le:YConstraint}
\end{lemma}

The identity \eqref{eq:XY} says that if we concatenate $\cY$ and $\cX$ to form the word $\cY \cX$ (of length $n+p$),
then flip the symbols for indices $0$ and $\alpha$, we obtain the concatenation of $\cY$ and $\cX$ in the reverse order.
Note that we always have $\alpha < n+p$ in view of \eqref{eq:S2}.

%...............................................................................
\begin{proof}
Choose any $k = 0,\ldots,n+p-1$.
Then $k = j n + i$ for unique values $j \in \mathbb{Z}$ and $i \in \{ 0,\ldots,n-1 \}$.
Since $\cS_0$ corresponds to $\cY_0$,
we have $(\cY \cX)_k = \cS_k = \cS_{j n + i}$,
and $(\cX \cY)_k = \cS_{k-n} = \cS_{(j-1) n + i}$.
By \eqref{eq:twoCrossingsAlt}, $(\cX \cY)_k = (\cY \cX)_k$ if and only if $k \ne 0, \alpha$.
Hence $\cX \cY = \left( \cY \cX \right)^{\overline{0} \, \overline{\alpha}}$ as required.
\end{proof}

The identity $\cX^{(d)} = \cX^{\overline{0} \, \overline{\hat{\imath}}}$
(where $\cX^{(d)}$ is the $d^{\rm th}$ left cyclic permutation of $\cX$)
follows immediately from \eqref{eq:twoCrossingsAlt}.
Here we show that this identity is more generally a consequence of \eqref{eq:XY}\removableFootnote{
It turns out that \eqref{eq:XY} also implies \eqref{eq:twoCrossings}.
Here is a proof that uses arguments in the proof of Lemma \ref{le:XConstraint}.

Choose any $j \in \mathbb{Z}$ and $i = 0,\ldots,n-1$.
By \eqref{eq:S}, if $j n + i < 0$, then $\cS_{j n + i} = \cX_i$,
as stated in \eqref{eq:twoCrossings}.
Moreover, if $j n + i \ge p$, then
$\cS_{j n + i} = \cX_{(j n + i - p) {\rm \,mod\,} n} = \cX_{(i+d) {\rm \,mod\,} n} = \cX^{(d)}_i$,
and so, by \eqref{eq:rotAlt}, $\cS_{j n + i} = \cX_i$ if and only if $i \ne 0, \hat{\imath}$,
as stated in \eqref{eq:twoCrossings}.
It remains to verify \eqref{eq:twoCrossings} for $0 \le j n + i < p$.
In this case, by \eqref{eq:S} we have $\cS_{j n + i} = \cY_{j n + i}$.
By following the arguments in the proof of Lemma \ref{le:XConstraint}, we obtain
$\cY_{j n + i} = (\cY \cX)_{k n + i} = \cX^{(d)}_i$ except with
$j < \hat{\jmath}$ and $i = \hat{\imath}$ because
$(\cX \cY)_{\hat{\jmath} n + \hat{\imath}} \ne (\cY \cX)_{\hat{\jmath} n + \hat{\imath}}$.
That is, in this case $\cS_{j n + i} = \cX^{(d)}_i$ unless
$j < \hat{\jmath}$ and $i = \hat{\imath}$ which,
in view of \eqref{eq:rotAlt}, completes our verification of \eqref{eq:twoCrossings}.
}.

%...............................................................................
\begin{lemma}
Let $\cX$ and $\cY$ be words of length $n$ and $p$ %with $\cX_0 \ne \cY_0$
that satisfy \eqref{eq:XY} for some $\alpha \in \{ 1,\ldots,n+p-1 \}$,
and let $d$ be given by \eqref{eq:d}.
Then
\begin{equation}
\cX^{(d)} = \cX^{\overline{0} \, \overline{\hat{\imath}}} \;,
\label{eq:rotAlt}
\end{equation}
where $\hat{\imath} = \alpha {\rm ~mod~} n$.
\label{le:XConstraint}
\end{lemma}

%...............................................................................
\begin{proof}
%Choose any $i \in \{ 0,\ldots,n-1 \}$ with $i \ne 0, \hat{\imath}$.
%Then $\cX_i = (\cX \cY)_i = (\cY \cX)_i$.
%If $i < p$ then $(\cY \cX)_i = \cY_i = (\cX \cY)_{n+i} = (\cY \cX)_{n+i}$.
%If $n+i < p$ then we repeat the previous step until we arrive
%at the index $j n + i$, where $j \in \mathbb{Z}$ is the smallest integer
%for which $j n + i \ge p$.
%This gives us $\cX_i = (\cY \cX)_{j n + i}$,
%and since $j n + i \ge p$ we have
%$(\cY \cX)_{j n + i} = \cX_{(j n + i - p) {\rm \,mod\,} n} =
%\cX_{(i+d) {\rm \,mod\,} n} = \cX^{(d)}_i$.
%Therefore $\cX_i = \cX^{(d)}_i$.
%With instead $i = 0$ or $i = \hat{\imath}$,
%the same arguments apply except that in exactly one equation
%an equals sign is replaced by a not-equals sign, and so $\cX_i \ne \cX^{(d)}_i$.
%Specifically, with $i=0$ we have $(\cX \cY)_i \ne (\cY \cX)_i$,
%and with $i = \hat{\imath}$ we have
%$(\cX \cY)_{\hat{\jmath} n + i} \ne (\cY \cX)_{\hat{\jmath} n + i}$.
%In summary, for any $i = 0,\ldots,n-1$, we have
%$\cX_i = \cX^{(d)}_i$ if and only if $i \ne 0, \hat{\imath}$,
%which verifies \eqref{eq:rotAlt}.
Choose any $i \in \{ 0,\ldots,n-1 \}$.
Then $\cX_i = (\cX \cY)_i = (\cY \cX)_i$ unless $i=0$ in which case $\cX_i = (\cX \cY)_i \ne (\cY \cX)_i$.
Let $k$ be the unique integer for which $p \le k n + i < p + n$.
From $(\cY \cX)_i$ we now repeatedly step $n$ places along $\cY \cX$ until we reach $(\cY \cX)_{k n + i}$ as follows.
For each $j = 1,\ldots,k$, we have
$(\cY \cX)_{(j-1) n + i} = \cY_{(j-1) n + i} = (\cX \cY)_{j n + i} = (\cY \cX)_{j n + i}$
unless $i = \hat{\imath}$ and $j = \hat{\jmath}$ (which arises because $\hat{\jmath} n + \hat{\imath} < p + n$)
in which case $(\cX \cY)_{j n + i} \ne (\cY \cX)_{j n + i}$.
Since $k n + i \ge p$, we have
$(\cY \cX)_{k n + i} = \cX_{k n + i - p} = \cX_{(i+d) {\rm \,mod\,} n} = \cX^{(d)}_i$.
We have thus shown that $\cX_i = \cX^{(d)}_i$ if and only if $i \notin \{ 0, \hat{\imath} \}$,
except in the special case $\hat{\imath} = 0$ for which we have shown that $\cX_i = \cX^{(d)}_i$ for all $i$.
This verifies \eqref{eq:rotAlt}.
%We have thus shown that $\cX_i = (\cY \cX)_{k n + i}$ if and only if $i \notin \{ 0, \hat{\imath} \}$.
%Since $k n + i \ge p$, we have
%$(\cY \cX)_{k n + i} = \cX_{k n + i - p} = \cX_{(i+d) {\rm \,mod\,} n} = \cX^{(d)}_i$.
%Thus $\cX_i = \cX^{(d)}_i$ if and only if $i \notin \{ 0, \hat{\imath} \}$,
%which verifies \eqref{eq:rotAlt}.
\end{proof}

Lastly we note that words satisfying \eqref{eq:rotAlt} are ``rotational'',
as defined below, if $d$ satisfies \eqref{eq:dConstraints}.
Such words are well-studied in the context of piecewise-linear maps
because they relate to rigid rotation on a circle \cite{SiMe09,Si16}.

%~~~~~~~~~~~~~~~~~~~~~~~~~~~~~~~~~~~~~~~~~~~~~~~~~~~~~~~~~~~~~~~~~~~~~~~~~~~~~~~
\begin{definition}
Let $\ell$, $m$ and $n$ be positive integers with
\begin{equation}
\ell < n \;, \qquad
m < n \;, \qquad
\gcd(m,n) = 1 \;.
%\label{eq:lmConstraints}
\nonumber
\end{equation}
Define a word $\cF[\ell,m,n]$ of length $n$ by
\begin{equation}
\cF[\ell,m,n] = \begin{cases}
L \;, & i m {\rm ~mod~} n < \ell \;, \\
R \;, & i m {\rm ~mod~} n \ge \ell \;,
\end{cases}
%\label{eq:rot}
\nonumber
\end{equation}
for $i = 0,\ldots,n-1$.
Words $\cF[\ell,m,n]$ and their cyclic permutations are called {\em rotational}.
\label{df:rot}
\end{definition}

%...............................................................................
\begin{lemma}
Let $\cX$ be a word of length $n$ with $\cX_0 = R$.
Suppose \eqref{eq:rotAlt} is satisfied for some $\hat{\imath} = 1,\ldots,n-1$
and $d$ satisfying \eqref{eq:dConstraints}.
Let $m$ denote the multiplicative inverse of $d$ modulo $n$
(i.e.~$m d {\rm ~mod~} n = 1$)
and let $\ell = \hat{\imath} m {\rm ~mod~} n$.
Then $\cX^{(d)} = \cF[\ell,m,n]$.
\label{le:rot}
\end{lemma}

For a proof of Lemma \ref{le:rot} see \cite{Si15c}\removableFootnote{
One could go further and say that a word is rotational if and only if
flipping two symbols yields a cyclic permutation with $\gcd(d,n) = 1$.
However, this is a bit awkward to say precisely given that
I am always assuming that the first symbol to be flipping is with $i = 0$.
Note: both definitions implicitly require $n > 1$.
}.
%For each of our examples in \S\ref{sec:examples}, the words $\cX$ satisfy \eqref{eq:dConstraints}
%and so by Lemma \ref{le:rot} these words are also rotational.

%=====================================================================
\section{Subsumed homoclinic connections}
\label{sec:subsumed}
\setcounter{equation}{0}

Let $\cX$ be a word of length $n$.
Suppose that $\lambda_1$ and $\lambda_2$ are eigenvalues of $M_{\cX}$ with multiplicity one.
Let $\beta$ denote the maximum modulus of the remaining eigenvalues of $M_{\cX}$ and suppose
\begin{equation}
0 \le \beta < \lambda_2 < 1 < \lambda_1 \;.
\label{eq:eigenvalueCondition}
\end{equation}
Since $1$ is not an eigenvalue of $M_{\cX}$, the $\cX$-cycle is unique.
If it is admissible with no points on $\Sigma$, then as a periodic solution of $f$
it has a one-dimensional unstable manifold $W^u \!\left( \left\{ x^{\cX}_i \right\} \right)$
and an $(N-1)$-dimensional stable manifold $W^s \!\left( \left\{ x^{\cX}_i \right\} \right)$.

We now consider these manifolds as they emanate from $x^{\cX}_0$.
For the affine map $f_{\cX}$,
let $E^u \!\left( x^{\cX}_0 \right)$ and $E^s \!\left( x^{\cX}_0 \right)$
denote the unstable and stable manifolds of its fixed point $x^{\cX}_0$.
Then for the piecewise-linear map $f$,
the part of $W^u \!\left( \left\{ x^{\cX}_i \right\} \right)$
emanating from $x^{\cX}_0$ coincides with $E^u \!\left( x^{\cX}_0 \right)$,
and the part of $W^s \!\left( \left\{ x^{\cX}_i \right\} \right)$
emanating from $x^{\cX}_0$ coincides with $E^s \!\left( x^{\cX}_0 \right)$.
Note that $W^u \!\left( \left\{ x^{\cX}_i \right\} \right)$
has two branches; these emanate from the $\cX$-cycle in opposite directions.

For $i = 1,2$, let $\omega_i^{\sf T}$ and $\zeta_i$ be left and right eigenvectors of $M_{\cX}$
corresponding to $\lambda_i$, and with
\begin{equation}
\omega_i^{\sf T} \zeta_i = 1 \;,
\label{eq:eigenvectorNormalisation}
\end{equation}
which can always be achieved because $\lambda_1$ and $\lambda_2$ have multiplicity one\removableFootnote{
This is an immediate consequence of the facts that
the span of all the generalised eigenspaces is $\mathbb{R}^N$
and that left and right eigenvectors are orthogonal for different eigenvalues.
However, here is a more direct argument:

Let $\lambda$ be an eigenvalue of $A$ with multiplicity one,
and let $\omega^{\sf T}$ and $\zeta$ be left and right eigenvectors.
Suppose for a contradiction that $\omega^{\sf T} \zeta = 0$.
Then since $\zeta$ is in the nullspace of $\lambda I - A$,
by the fundamental theorem of linear algebra
$\omega$ is in the range of $(\lambda I - A)^{\sf T}$.
That is, there exists $u \in \mathbb{R}^N$ such that 
\begin{equation}
\omega^{\sf T} = u^{\sf T} (\lambda I - A) \;.
\label{eq:eigenvectorNormalisationProof}
\end{equation}
Multiplying both sides of \eqref{eq:eigenvectorNormalisationProof}
by $(\lambda I - A)$ gives ${\bf 0} = u^{\sf T} (\lambda I - A)^2$.
Then since the multiplicity of $\lambda$ is one,
$u$ must be a scalar multiple of $\omega$.
Substituting $u = k \omega$ into \eqref{eq:eigenvectorNormalisationProof} gives
$\omega^{\sf T} = k \omega^{\sf T} (\lambda I - A) = {\bf 0}$,
which is a contradiction.
}.
Then $E^u \!\left( x^{\cX}_0 \right)$ has direction $\zeta_1$
and $E^s \!\left( x^{\cX}_0 \right)$ includes direction $\zeta_2$.
Moreover, $\omega_1^{\sf T} \zeta_2 = 0$ and $\omega_2^{\sf T} \zeta_1 = 0$
(because $\lambda_1 \ne \lambda_2$ and with $i,j \in \{ 1,2 \}$ we have
$\left( \lambda_i \omega_i^{\sf T} \right) \zeta_j
= \omega_i^{\sf T} M_{\cX} \zeta_j
= \omega_i^{\sf T} \left( \lambda_j \zeta_j \right)$
and so $(\lambda_i - \lambda_j) \omega_i^{\sf T} \zeta_j = 0$).
Thus $E^u \!\left( x^{\cX}_0 \right)$ is orthogonal to $\omega_2$
and $E^s \!\left( x^{\cX}_0 \right)$ is orthogonal to $\omega_1$.
Since $\beta < \lambda_2$, direction $\zeta_2$ represents
the slowest direction for the dynamics within $E^s \!\left( x^{\cX}_0 \right)$.

If $e_1^{\sf T} \zeta_1 \ne 0$, then $E^u \!\left( x^{\cX}_0 \right)$
intersects $\Sigma$ at a unique point (denoted $y_0$ below), see Fig.~\ref{fig:InvManSchem}.

%%%%%%%%%%%%%%%%%%%%%%%%%%%%%%%%%%%%%%%%%%%%%%%%%%%%%%%%%%%%%
\begin{figure}[b!]
\begin{center}
\setlength{\unitlength}{1cm}
\begin{picture}(8,6)
\put(0,0){\includegraphics[height=6cm]{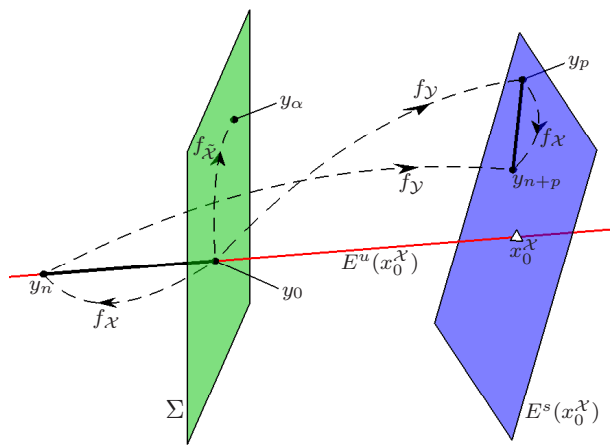}}
\put(6.64,2.64){\scriptsize $x^{\cX}_0$}
\put(4.4,2.49){\scriptsize $E^u(x^{\cX}_0)$}
\put(6.84,.5){\scriptsize $E^s(x^{\cX}_0)$}
\put(2.07,.5){\footnotesize $\Sigma$}
\put(3.62,2.13){\scriptsize $y_0$}
\put(.26,2.2){\scriptsize $y_n$}
\put(3.59,4.66){\scriptsize $y_{\alpha}$}
\put(7.42,5.2){\scriptsize $y_p$}
\put(6.68,3.62){\scriptsize $y_{n+p}$}
\put(1.1,1.75){\scriptsize $f_{\cX}$}
\put(5.36,4.82){\scriptsize $f_{\cY}$}
\put(2.43,4.05){\scriptsize $f_{\hspace{-.35mm}\tilde{\cX}}$}
\put(7.04,4.16){\scriptsize $f_{\cX}$}
\put(5.17,3.56){\scriptsize $f_{\cY}$}
\end{picture}
\caption{
A sketch of the phase space of $f$ relating to Proposition \ref{pr:subsumed}.
The manifolds $E^s \big( x^{\cX}_0 \big)$ and $E^u \big( x^{\cX}_0 \big)$ are
the stable and unstable manifolds of $x^{\cX}_0$ for $f_{\cX}$.
We show several important points of the homoclinic $\cS$-orbit $\{ y_i \}$
and how they are arrived at from $y_0$ under compositions of the components of $f$,
where $\tilde{\cX}$ denotes the first $\alpha$ symbols of $\cX \cY$.
\label{fig:InvManSchem}
}
\end{center}
\end{figure}
%%%%%%%%%%%%%%%%%%%%%%%%%%%%%%%%%%%%%%%%%%%%%%%%%%%%%%%%%%%%%

%...............................................................................
\begin{proposition}
Suppose
\vspace{-2mm}																					% <-- manual spacing!!
\begin{enumerate}
\setlength{\itemsep}{0pt}
\item
$\cX$ and $\cY$ satisfy
$\cX \cY = \left( \cY \cX \right)^{\overline{0} \, \overline{\alpha}}$
for some $\alpha = 1,\ldots,n+p-1$;
\label{it:subsumedXY}
\item
the eigenvalues of $M_{\cX}$ satisfy \eqref{eq:eigenvalueCondition},
and $e_1^{\sf T} \zeta_1 \ne 0$;
\label{it:subsumedEigenvalues}
\item
the $\cX$-cycle is admissible with $x^{\cX}_0 \notin \Sigma$\removableFootnote{
Otherwise $y_i = x^{\cX}_0$ for all $i \in \mathbb{Z}$
which, rather oddly, doesn't appear to invalidate any of the other assumptions.

It is interesting that the proof doesn't require $x^{\cX}_i \notin \Sigma$ for all $i$.
};
\label{it:subsumedXcycle}
\item
there exists an admissible $\cS$-orbit, denoted $\{ y_i \}$ and where $\cS$ is given by \eqref{eq:S},
that is homoclinic to the $\cX$-cycle and
with $y_0 = E^u \!\left( x^{\cX}_0 \right) \cap \Sigma$ and $y_{\alpha} \in \Sigma$.
\label{it:subsumedSorbit}
\end{enumerate}
Then
\begin{equation}
W^u_{\rm br} \!\left( \left\{ x^{\cX}_i \right\} \right) =
\bigcup_{i \in \mathbb{Z}} \cL[y_i,y_{i+n})
\subset W^s \!\left( \left\{ x^{\cX}_i \right\} \right),
%\label{eq:Wubranch}
\nonumber
\end{equation}
where $W^u_{\rm br} \!\left( \left\{ x^{\cX}_i \right\} \right)$ denotes the branch
of $W^u \!\left( \left\{ x^{\cX}_i \right\} \right)$ that contains the $\cS$-orbit.
\label{pr:subsumed}
\end{proposition}

Proposition \ref{pr:subsumed} tells us that if $f$ has a homoclinic $\cS$-orbit $\{ y_i \}$,
with $y_0, y_{\alpha} \in \Sigma$,
then, under certain conditions, the branch of $W^u \!\left( \left\{ x^{\cX}_i \right\} \right)$ containing $\{ y_i \}$
is given simply by connecting each $y_i$ to $y_{i+n}$ by a line segment,	
and that this branch is contained within $W^s \!\left( \left\{ x^{\cX}_i \right\} \right)$
and so the homoclinic connection is subsumed.

Proposition \ref{pr:subsumed} is proved below by first using Lemma \ref{le:follow}
to show that since $y_0, y_{\alpha} \in \Sigma$, both $y_0$ and $y_n$ follow $\cY \cX^{\infty}$ under $f$
and follow $\cX^{\infty}$ under backwards iteration of $f$.
It follows that the endpoints of each interval $\left[ y_i, y_{i+n} \right]$ cannot lie on different sides of $\Sigma$,
and the result is then a consequence of Lemma \ref{le:lineSegment}.

%...............................................................................
\begin{proof}[Proof of Proposition \ref{pr:subsumed}]

By assumption, $y_0$ follows $\cY \cX^{\infty}$ under $f$
(meaning $y_0$ follows $\cY \cX^k$ under $f$ for all positive $k$).
Since $y_0 \in \Sigma$ and $y_{\alpha} \in \Sigma$,
by Lemma \ref{le:follow} $y_0$ also follows $(\cY \cX)^{\overline{0} \, \overline{\alpha}} \cX^{\infty}$.
Thus by \eqref{eq:XY}, $y_0$ follows $\cX \cY \cX^{\infty}$.
Hence $y_n$ follows $\cY \cX^{\infty}$.

In particular, we have shown that $y_0$ follows $\cX$.
Since the $\cX$-cycle is admissible, $x^{\cX}_0$ also follows $\cX$.
Thus by Lemma \ref{le:lineSegment}, every $x \in \cL \!\left[ x^{\cX}_0, y_0 \right]$ follows $\cX$.
Note that $\cL \!\left[ x^{\cX}_0, y_0 \right]$ does not consist of a single point
because $y_0 \in \Sigma$ and $x^{\cX}_0 \notin \Sigma$ and so $y_0 \ne x^{\cX}_0$.
We also have $y_{-n}, y_{-2 n}, \ldots \in \cL \!\left[ x^{\cX}_0, y_0 \right]$
with $y_{j n} \to x^{\cX}_0$ as $j \to -\infty$.

Since $f_{\cX}$ is affine, $W^u_{\rm br} \!\left( \left\{ x^{\cX}_i \right\} \right)$
is linear in a neighbourhood\removableFootnote{
Note that $W^u \!\left( \left\{ x^{\cX}_i \right\} \right)$ is not necessarily
linear in a neighbourhood if $x_i \in \Sigma$ for some $i$, I think.
}
of the $\cX$-cycle.
Since $y_{j n} \to x^{\cX}_0$ as $j \to -\infty$, there exists a constant $C \in \mathbb{Z}$ such that
\begin{equation}
W^u_{\rm br} \!\left( \left\{ x^{\cX}_i \right\} \right) =
\bigcup_{i = 0,1,\ldots} f^i \!\left( \cL \! \left( x^{\cX}_0, y_{j n} \right) \right),
\label{eq:Wubranch2}
\end{equation}
for all $j \le C$.
Since every $x \in \cL \!\left[ x^{\cX}_0, y_0 \right]$ follows $\cX$,
we can use $j = 1$ in \eqref{eq:Wubranch2}.
Thus, in particular, $\cL[y_0,y_n) \subset W^u_{\rm br} \!\left( \left\{ x^{\cX}_i \right\} \right)$, see Fig.~\ref{fig:InvManSchem}.
To then describe $W^u_{\rm br} \!\left( \left\{ x^{\cX}_i \right\} \right)$ more succinctly,
we observe that $\cL[y_0,y_n)$ is a {\em fundamental domain} in that every
orbit of $f$ in $W^u_{\rm br} \!\left( \left\{ x^{\cX}_i \right\} \right)$ involves exactly one point in $\cL[y_0,y_n)$.
For all $i < 0$\removableFootnote{
We take care to avoid discussing preimages because $f$ may not be invertible.
},
we have
\begin{equation}
f \!\left( \cL[y_{i-1},y_{i-1+n}) \right) = \cL[y_i,y_{i+n}) \;,
\label{eq:fLi}
\end{equation}
because the components of $f$ are affine and each $y_0, y_{-n}, \ldots$ follows $\cX$.
Equation \eqref{eq:fLi} also holds for all $i \ge 0$
because $y_0$ and $y_n$ both follow $\cY \cX^{\infty}$.
Therefore $W^u_{\rm br} \!\left( \left\{ x^{\cX}_i \right\} \right) =
\bigcup_{i \in \mathbb{Z}} \cL[y_i,y_{i+n})$.
The line segments $\cL[y_i,y_{i+n})$ converge to the $\cX$-cycle as $i \to \infty$,
thus the forward orbit of every point in $W^u_{\rm br} \!\left( \left\{ x^{\cX}_i \right\} \right)$
converges to the $\cX$-cycle, and therefore
$W^u_{\rm br} \!\left( \left\{ x^{\cX}_i \right\} \right) \subset W^s \!\left( \left\{ x^{\cX}_i \right\} \right)$.
\end{proof}

%=====================================================================
\section{Main results}
\label{sec:thms}
\setcounter{equation}{0}

Here we state the main results and then discuss the assumptions contained within them.
The results concern $\cX^k \cY$-cycles where $k \ge 0$.
These are periodic solutions that follow $\cX$ a total of $k$ times, then follow $\cY$, and then repeat.
Theorem \ref{th:infinity} provides sufficient conditions
for the existence of admissible, asymptotically stable $\cX^k \cY$-cycles for infinitely many values of $k \ge 0$.
Theorem \ref{th:homoclinic} indicates that many of the conditions of Theorem \ref{th:infinity} are necessary.
Moreover, Proposition \ref{pr:subsumed} tells us that the conditions of Theorem \ref{th:infinity}
imply the presence of a subsumed homoclinic connection.
The nature of $\cX^k \cY^{\overline{0}}$-cycles is also included in Theorem \ref{th:homoclinic}.
These periodic solutions are unstable;
numerical computations suggest that their stable manifolds typically form the
boundaries between the basins of attraction of the $\cX^k \cY$-cycles.

%...............................................................................
\begin{theorem}
Suppose
\vspace{-2mm}																					% <-- manual spacing!!
\begin{enumerate}
\setlength{\itemsep}{0pt}
\item
$\cX$ and $\cY$ satisfy
$\cX \cY = \left( \cY \cX \right)^{\overline{0} \, \overline{\alpha}}$
for some $\alpha = 1,\ldots,n+p-1$\removableFootnote{
This condition implies $\cX_0 \ne \cY_0$.
We do not seem need to assume that ${\rm gcd}(n,p) = 1$
and so it may be possible that this codimension-three scenario
is possible for non-rotational words $\cX$.
};
\label{it:infinityXY}
\item
the eigenvalues of $M_{\cX}$ satisfy \eqref{eq:eigenvalueCondition},
and $e_1^{\sf T} \zeta_1 \ne 0$;
\label{it:infinityEigenvalues}
\item
$\lambda_1 \lambda_2 = 1$ and $\lambda_2 < c < 1$, where
\begin{equation}
c = \det \!\left(
\begin{bmatrix} \omega_1^{\sf T} \\ \omega_2^{\sf T} \end{bmatrix} M_{\cY}
\begin{bmatrix} \zeta_1 & \zeta_2 \end{bmatrix} \right);
\label{eq:c}
\end{equation}
\label{it:infinitylam1lam2c}
\item
the $\cX$-cycle is admissible with no points on $\Sigma$;
\label{it:infinityXcycle}
\item
there exists an admissible $\cS$-orbit, denoted $\{ y_i \}$ and where $\cS$ is defined by \eqref{eq:S},
that is homoclinic to the $\cX$-cycle and
with $y_0 = E^u \!\left( x^{\cX}_0 \right) \cap \Sigma$ and $y_{\alpha} \in \Sigma$;
\label{it:infinitySorbit}
\item
there does not exist $i \ge 0$ such that
$y_i \in \Sigma$ and $y_{i+n} \in \Sigma$.
\label{it:infinityyi}
\end{enumerate}
Then there exists $k_{\rm min} \ge 0$
such that for all $k \ge k_{\rm min}$
the map $f$ has an admissible, asymptotically stable $\cX^k \cY$-cycle
with no points on $\Sigma$.
\label{th:infinity}
\end{theorem}

%...............................................................................
\begin{theorem}
Suppose
\vspace{-2mm}																					% <-- manual spacing!!
\begin{enumerate}
\setlength{\itemsep}{0pt}
\item
$\cX$ and $\cY$ satisfy
$\cX \cY = \left( \cY \cX \right)^{\overline{0} \, \overline{\alpha}}$
for some $\alpha = 1,\ldots,n+p-1$;
\label{it:homoclinicXY}
\item
the eigenvalues of $M_{\cX}$ satisfy \eqref{eq:eigenvalueCondition},
and $e_1^{\sf T} \zeta_1 \ne 0$;
\label{it:homoclinicEigenvalues}
\item
$c \notin \{ -1,0 \}$, where $c$ is given by \eqref{eq:c};
%$c \ne -1$ and $c \ne 0$, where $c$ is given by \eqref{eq:c};
\label{it:homoclinicc}
\item
$f_{\cY^{\overline{0}}} \left( E^u \!\left( x^{\cX}_0 \right) \right) \not\subset
E^s \!\left( x^{\cX}_0 \right)$\removableFootnote{
Since $f_{\cY^{\overline{0}}} \left( E^u \!\left( x^{\cX}_0 \right) \right)$
is a line that we know intersects $E^s \!\left( x^{\cX}_0 \right)$
(because $y_0 \in E^u \!\left( x^{\cX}_0 \right)$
and $f_{\cY^{\overline{0}}}(y_0) \in E^s \!\left( x^{\cX}_0 \right)$),
an equivalent condition is
that there exists $x \in E^u \!\left( x^{\cX}_0 \right)$
such that $f_{\cY^{\overline{0}}}(x) \not\in E^s \!\left( x^{\cX}_0 \right)$,
and yet another equivalent condition is
$f_{\cY^{\overline{0}}} \left( x^{\cX}_0 \right) \not\in E^s \!\left( x^{\cX}_0 \right)$.
};
\label{it:homoclinicY0bar}
\item
there exists $k_{\rm min} \ge 0$
such that for all $k \ge k_{\rm min}$
the map $f$ has an admissible, stable $\cX^k \cY$-cycle
with no points on $\Sigma$\removableFootnote{
Why can't we have points on $\Sigma$?
Answer: otherwise we do not know how do derive consequences of stability in Step 1.
}.
\label{it:homoclinicInfinity}
\end{enumerate}
Then the $\cX$-cycle is admissible with $x^{\cX}_0 \notin \Sigma$\removableFootnote{
We cannot prove that the $\cX$-cycle has no points on $\Sigma$.
The point $x^{\cX}_0$ is special because the $\cX_0 \ne \cY_0$.
}
and there exists an $\cS$-orbit (not necessarily admissible),
denoted $\{ y_i \}$ and where $\cS$ is defined by \eqref{eq:S},
that is homoclinic to the $\cX$-cycle and
with $y_0 = E^u \!\left( x^{\cX}_0 \right) \cap \Sigma$
and $y_{\alpha} \in \Sigma$\removableFootnote{
We cannot show that the $\cS$-orbit is admissible.
For a counterexample imagine perturbing the example of Fig.~\ref{fig:qqExdRnonzero9p1}
so that $y_1$ crosses $\Sigma$.

Similarly we cannot show that there does not exist $i \ge 0$ such that
$y_i \in \Sigma$ and $y_{i+n} \in \Sigma$.
}.
Moreover, $\lambda_1 \lambda_2 = 1$ and $\lambda_2 \le c \le 1$.
Also, $\cX^k \cY^{\overline{0}}$-cycles are unique (although not necessarily admissible)
for sufficiently large values of $k$, and $x^{\cX^k \cY^{\overline{0}}}_{k n} \to y_0$ as $k \to \infty$.
\label{th:homoclinic}
\end{theorem}

We conclude this section with some technical remarks about the theorem statements.
Proofs of the theorems are deferred to \S\ref{sec:proofs}.
\begin{itemize}
\item[\raisebox{.6mm}{\tiny $\bullet$}]
The scenario described in Theorem \ref{th:infinity} is codimension-three because
each of the conditions
%$y_0 = E^u \!\left( x^{\cX}_0 \right) \cap \Sigma$,
$y_0 \in W^s \!\left( \left\{ x^{\cX}_i \right\} \right)$,
$y_{\alpha} \in \Sigma$ and
$\lambda_1 \lambda_2 = 1$ is codimension-one.
\item[\raisebox{.6mm}{\tiny $\bullet$}]
Assumptions \eqref{it:infinityXcycle} and \eqref{it:infinityyi} of Theorem \ref{th:infinity}
are included in order to prove admissibility of the $\cX^k \cY$-cycles.
If $x^{\cX}_i \in \Sigma$, for some $i = 1,\ldots,n$,
then it may be possible for points of an $\cX^k \cY$-cycle near $x^{\cX}_i$
to lie on the wrong side of $\Sigma$\removableFootnote{
I think that the $\cX^k \cY$-cycles must be admissible if either
the $\cX$-cycle has no points on $\Sigma$ or $e_1^{\sf T} \zeta_2 \ne 0$
(both conditions occur generically)
but that this would be icky to prove.
Moreover I think that the given condition is necessary 
because I can envisage a codimension-five counterexample involving $e_1^{\sf T} \zeta_2 = 0$.
}.
If $y_i$ and $y_{i+n}$ both lie in $\Sigma$, for some $i \ge 0$,
then our method of proof of admissibility fails because
any fattening of $\cL(y_i,y_{i+n})$ into an open set (containing one point of each $\cX^k \cY$-cycle,
for sufficiently large $k$) will not lie entirely on one side of $\Sigma$\removableFootnote{
I suspect that this condition is necessary
as I can envisage a plausible codimension-five scenario
(or a codimension-four scenario using $i = 0$) in which there exists such an $i$
and $\cX^k \cY$-cycles each involve one point on the wrong side of $\Sigma$.
}.
\item[\raisebox{.6mm}{\tiny $\bullet$}]
As evident in the proof of Theorem \ref{th:infinity}, the assumption
$c > \lambda_2$ ensures that $\cX^k \cY$-cycles are admissible
whereas $c < 1$ ensures that they are asymptotically stable.
\item[\raisebox{.6mm}{\tiny $\bullet$}]
In Theorem \ref{th:homoclinic}, the assumptions $c \ne -1$ and $c \ne 0$
eliminate the possibility of alternate geometric structures
that have codimension greater than three\removableFootnote{
Our method of proof is to study the properties of $g_{\cY}$.
If $c = -1$ or $c = 0$ then we cannot rule out the possibility that
$g_{\cY}$ satisfies some unusual criteria
that do not produce the required properties.
Roughly, if $c = 0$ then the resulting homoclinic connection
may approach $x^{\cX}_0$ in a direction other than $\zeta_2$,
and if $c = -1$ then as $k \to \infty$ the points of the $\cX^k \cY$-cycles
could be given by a $0/0$-type limit.
Yet certainly, each condition $c = -1$ and $c = 0$ is either necessary
or will require completely different mathematical arguments
(that are likely to be very difficult) to show that it is unnecessary.
}.
\item[\raisebox{.6mm}{\tiny $\bullet$}]
Assumption \eqref{it:homoclinicY0bar} of Theorem \ref{th:homoclinic} is
used in our proof to show that $y_{\alpha} \in \Sigma$\removableFootnote{
Although I believe that this condition is unnecessary.
My reasoning is that one can show that
\begin{equation}
\omega_1^{\sf T} \xi_{\tilde{\cY}} e_1^{\sf T} f_{\tilde{\cX}} \left( x^{\cX}_0 + c \zeta_1 \right) =
\lambda_1 \omega_1^{\sf T} \xi_{\cY} e_1^{\sf T} \left( x^{\cX}_0 + c \zeta_1 \right) \;.
\label{eq:Y0barIdentity}
\end{equation}
The given assumption is equivalent to $\omega_1^{\sf T} \xi_{\cY} \ne 0$,
from which \eqref{eq:Y0barIdentity} quickly leads to the conclusion
that $e_1^{\sf T} y_{\alpha} = 0$, where
$y_{\alpha} = f_{\tilde{\cX}} \left( x^{\cX}_0 + c_0 \zeta_1 \right) = 0$.
All quantities in \eqref{eq:Y0barIdentity}
are multivariate polynomial functions of the parameters of $f$,
and the idea is that if $\omega_1^{\sf T} \xi_{\cY} = 0$
then one could change the parameters by an arbitrarily small amount to get
$\omega_1^{\sf T} \xi_{\cY} \ne 0$ and so by continuity
we must always have $e_1^{\sf T} y_{\alpha} = 0$.
One major difficulty is that the argument requires knowing that, given $\cX$ and $\cY$,
$\omega_1^{\sf T} \xi_{\cY} \ne 0$ for some combination of parameter values.
}
and to establish the given properties regarding $\cX^k \cY^{\overline{0}}$-cycles.
\end{itemize}

%=====================================================================
\section{Examples for the three-dimensional border-collision normal form}
\label{sec:examples}
\setcounter{equation}{0}

%The dynamics of a piecewise-smooth map can change in a fundamental way
%when an invariant set collides with a switching manifold as the parameters are varied.
%Piecewise-smooth maps exhibit novel bifurcations caused by the interactions between invariant sets and switching manifolds.
For piecewise-smooth maps, interactions between invariant sets and switching manifolds
give rise to so-called discontinuity-induced bifurcations \cite{DiBu08}.
The simplest bifurcation of this type corresponds to the collision of a fixed point with a switching manifold.
If the map is continuous and the derivatives of its components are bounded, at least locally,
then the bifurcation is referred to as a border-collision bifurcation.

%For instance, a stable fixed point can bifurcate (instantaneously) to a chaotic attractor %\cite{DiGa98,YuBa98},
%or to multiple attractors.
%The creation of multiple attractors in this fashion is particularly interesting because the attractors grow out of a single point.
%Also, multiple attractors can be created simultaneously.
%These grow out of a single point and 
%their closeness can create an inherent uncertainty in the long-term dynamics in the presence of noise \cite{DuNu99}.
%If the map is continuous with components that have bounded derivatives in a neighbourhood of the bifurcation,
%then the local dynamics is governed by a map of the form \eqref{eq:f}
%and the bifurcation is referred to as a {\em border-collision bifurcation}, \cite{Si16}.

%The collision of a fixed point with a switching manifold in a piecewise-smooth map
%is often a bifurcation responsible for the creation of new dynamics.
%Local properties of the bifurcation can be understood from an approximation to the map
%formed by expanding each relevant smooth piece of the map about the bifurcation and truncating to leading order.
%A variety of important classes of piecewise-smooth maps can be created in this manner \cite{DiBu08}.
%Arguably the simplest case is for the approximation to be a two-piece, piecewise-linear continuous map.
%In this case the bifurcation is known as a border-collision bifurcation,
%and if a certain non-degeneracy condition is satisfied then the approximation can be transformed to a normal form \cite{Si16}.

Diverse dynamics may be created in border-collision bifurcations, see \cite{Si16} for a recent review.
These dynamics are described by maps of the form \eqref{eq:f}.
If certain non-degeneracy conditions are satisfied,
then the map can be put into the border-collision normal form \cite{Di03}.
In three dimensions, this normal form can be written as
\begin{equation}
x_{i+1} =
\begin{cases}
\begin{bmatrix}
\tau_L & 1 & 0 \\
-\sigma_L & 0 & 1 \\
\delta_L & 0 & 0
\end{bmatrix}
x_i + 
\begin{bmatrix}
1 \\ 0 \\ 0
\end{bmatrix} \mu \;, & e_1^{\sf T} x_i \le 0 \;, \\
\begin{bmatrix}
\tau_R & 1 & 0 \\
-\sigma_R & 0 & 1 \\
\delta_R & 0 & 0
\end{bmatrix} 
x_i + 
\begin{bmatrix}
1 \\ 0 \\ 0
\end{bmatrix} \mu \;, & e_1^{\sf T} x_i \ge 0 \;,
\end{cases}
\label{eq:bcnf}
\end{equation}
where $\tau_L, \sigma_L, \delta_L, \tau_R, \sigma_R, \delta_R \in \mathbb{R}$
and $\mu \in \mathbb{R}$ is a parameter that controls the border-collision bifurcation which occurs at $\mu = 0$.
Since the structure of the dynamics of \eqref{eq:bcnf} is independent of the magnitude of $\mu$,
it is appropriate to set $\mu = 1$ in order to study dynamics created in border-collision bifurcations.
With this restriction, the parameter space of \eqref{eq:bcnf} is six-dimensional.

Given suitable words $\cX$ and $\cY$,
we can search for codimension-three points in the parameter space of \eqref{eq:bcnf} at which the map
has infinitely many asymptotically stable $\cX^k \cY$-cycles.
To do this, we first identify a computationally convenient set of three codimension-one
conditions that describe the scenario of Theorems \ref{th:infinity} and \ref{th:homoclinic}.

With a subsumed homoclinic connection, $f_{\cY}$ maps $E^u \!\left( x^{\cX}_0 \right)$
onto $E^s \!\left( x^{\cX}_0 \right)$, see Fig.~\ref{fig:InvManSchem}.
To describe this restriction algebraically,
we consider the scalar quantity $a = \omega_1^{\sf T} \!\left( x - x^{\cX}_0 \right)$
which is a measure of the displacement of $x \in \mathbb{R}^N$ from $E^s \!\left( x^{\cX}_0 \right)$.
Since $f_{\cY}$ is affine, 
this displacement changes affinely under $f{\cY}$.
That is, $\omega_1^{\sf T} \!\left( f_{\cY}(x) - x^{\cX}_0 \right) = \gamma_{11} a + \psi_1$,
for some constants $\gamma_{11}, \psi_1 \in \mathbb{R}$
(using the notation of \S\ref{sec:proofs}).
Thus with a subsumed homoclinic connection we must have $\gamma_{11} = \psi_1 = 0$.

The scenario of Theorems \ref{th:infinity} and \ref{th:homoclinic}
also requires $\lambda_1 \lambda_2 = 1$.
The three examples given below were obtained by fixing the values of
$\sigma_L$, $\sigma_R$ and $\delta_R$ and solving
\begin{equation}
\gamma_{11} = 0 \;, \qquad
\psi_1 = 0 \;, \qquad
\lambda_1 \lambda_2 = 1 \;,
\label{eq:threeConditions}
\end{equation}
for $\tau_L$, $\tau_R$ and $\delta_L$.
This is an effective approach because each of the constants in \eqref{eq:threeConditions}
can be evaluated numerically in a relatively efficient manner at
any point in the parameter space of \eqref{eq:bcnf}.
Once such a point is found,
its validity can be verified by checking each of the conditions in Theorems \ref{th:infinity} and \ref{th:homoclinic}.

First we consider
\begin{equation}
\cX = RLLR \;, \qquad
\cY = LLR \;.
\label{eq:XY9p1}
\end{equation}
For these words, infinitely many asymptotically stable $\cX^k \cY$-cycles were identified in \cite{Si14}
for the two-dimensional border-collision normal form.
Using the approach described above,
for the three-dimensional system \eqref{eq:bcnf} we obtained the values
\begin{equation}
\begin{gathered}
\tau_L = 1.1770635074 \;, \qquad
\sigma_L = 1 \;, \qquad
\delta_L = 0.4334058651 \;, \\
\tau_R = -1.0170722063 \;, \qquad
\sigma_R = 0.5 \;, \qquad
\delta_R = 1 \;,
\end{gathered}
\label{eq:param9p1}
\end{equation}
given to ten decimal places.

%%%%%%%%%%%%%%%%%%%%%%%%%%%%%%%%%%%%%%%%%%%%%%%%%%%%%%%%%%%%%
\begin{figure}[b!]
\begin{center}
\setlength{\unitlength}{1cm}
\begin{picture}(16,8)
\put(0,0){\includegraphics[width=16cm]{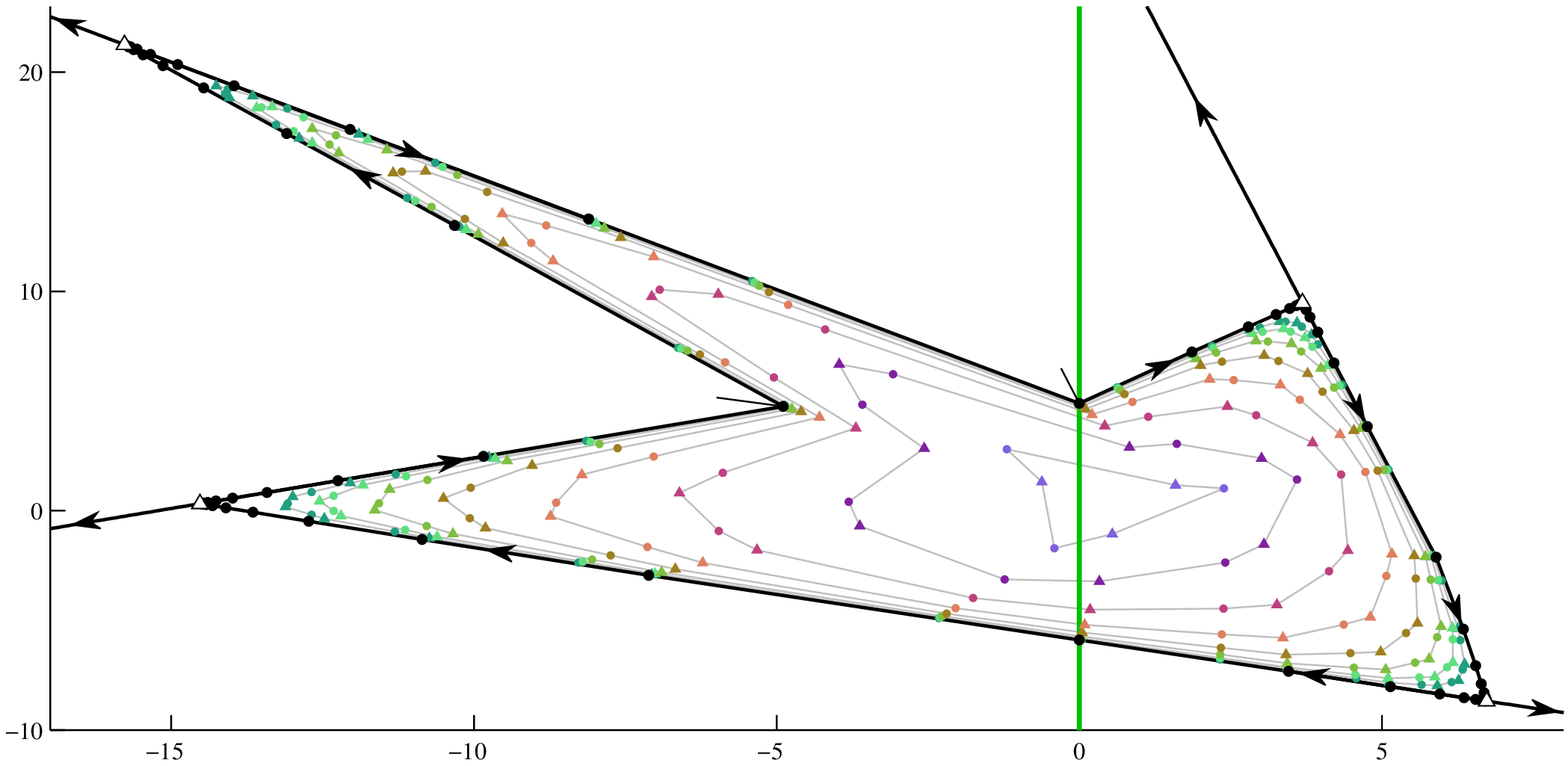}}
\put(8,0){\small $e_1^{\sf T} x$}
\put(0,4.2){\small $e_2^{\sf T} x$}
\put(15.26,1.28){\footnotesize $x^{\cX}_0$}
\put(2.37,2.69){\footnotesize $x^{\cX}_1$}
\put(1.89,7.72){\footnotesize $x^{\cX}_2$}
\put(13.48,5.1){\footnotesize $x^{\cX}_3$}
\put(10.87,6.5){\small $\Sigma$}
\put(13.03,1.3){\scriptsize $y_{\hspace{.8mm}4}$}
\put(13.18,1.278){\tiny -}
\put(5.2,3.73){\scriptsize $y_{\hspace{.8mm}3}$}
\put(5.35,3.708){\tiny -}
\put(6.25,6.06){\scriptsize $y_{\hspace{.8mm}2}$}
\put(6.4,6.038){\tiny -}
\put(14.09,3.89){\scriptsize $y_{\hspace{.8mm}1}$}
\put(14.24,3.868){\tiny -}
\put(10.88,1.62){\scriptsize $y_0$}
\put(7.34,4.14){\scriptsize $y_1$}
\put(10.82,4.54){\scriptsize $y_2$}
\put(14.77,2.63){\scriptsize $y_3$}
\put(6.85,2.22){\scriptsize $y_4$}
%\put(0,0){\scriptsize $y_{\hspace{.8mm}1}$}
%\put(.15,-.022){\tiny -}
\end{picture}
\caption{
A phase portrait of \eqref{eq:bcnf} with \eqref{eq:param9p1}.
We show the $\cX$-cycle and the homoclinic $\cS$-orbit $\{ y_i \}$,
where $\cX$ and $\cY$ are given by \eqref{eq:XY9p1}.
We use circles to show asymptotically stable $\cX^k \cY$-cycles for $k = 0,\ldots,7$,
and triangles to show unstable $\cX^k \cY^{\overline{0}}$-cycles for the same values of $k$.
\label{fig:qqExdRnonzero9p1}
}
\end{center}
\end{figure}
%%%%%%%%%%%%%%%%%%%%%%%%%%%%%%%%%%%%%%%%%%%%%%%%%%%%%%%%%%%%%

Fig.~\ref{fig:qqExdRnonzero9p1} shows the first two coordinates of the phase space
of \eqref{eq:bcnf} with \eqref{eq:param9p1}.
One branch of $W^u \!\left( \left\{ x^{\cX}_i \right\} \right)$
appears to converge to the $\cX$-cycle, thus forming a subsumed homoclinic connection.
Strictly speaking the homoclinic connection is not (quite) subsumed
because the parameter values are only accurate to ten decimal places.
The part of $W^u \!\left( \left\{ x^{\cX}_i \right\} \right)$ that emanates from $x^{\cX}_0$ converges to $x^{\cX}_1$.
This is because $d = 1$, where $d$ is given by \eqref{eq:d} with $n = 4$ and $p = 3$.
%(the lengths of $\cX$ and $\cY$).
Also the $\cS$-orbit $\{ y_i \}$ has $y_2 \in \Sigma$ because $\alpha = 2$.

The $\cX^k \cY$-cycles are admissible and asymptotically stable for at least $k = 0,1,\ldots,7$,
as shown in Fig.~\ref{fig:qqExdRnonzero9p1}.
The corresponding $\cX^k \cY^{\overline{0}}$-cycles are also shown
and approach the $\cS$-orbit with increasing values of $k$, as predicted by Theorem \ref{th:homoclinic}.
In order to distinguish the different periodic solutions clearly,
we have connected the corresponding pairs of $\cX^k \cY$ and $\cX^k \cY^{\overline{0}}$-cycles with line segments
(these line segments do not relate to the dynamics of \eqref{eq:bcnf}).

%%%%%%%%%%%%%%%%%%%%%%%%%%%%%%%%%%%%%%%%%%%%%%%%%%%%%%%%%%%%%
\begin{figure}[t!]
\begin{center}
\setlength{\unitlength}{1cm}
\begin{picture}(8,6)
\put(0,0){\includegraphics[height=6cm]{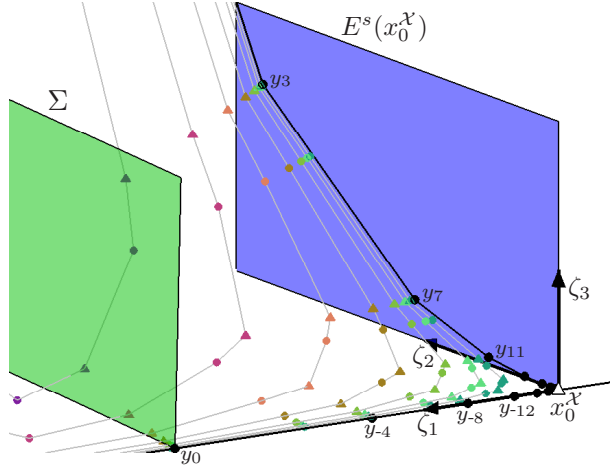}}
\put(7.18,.51){\footnotesize $x^{\cX}_0$}
\put(5.41,.29){\footnotesize $\zeta_1$}
\put(5.35,1.2){\footnotesize $\zeta_2$}
\put(7.42,2.1){\footnotesize $\zeta_3$}
\put(.5,4.57){\small $\Sigma$}
\put(4.4,5.57){\footnotesize $E^s(x^{\cX}_0)$}
\put(6.5,.53){\scriptsize $y_{\hspace{.8mm}12}$}
\put(6.65,.518){\tiny -}
\put(5.94,.44){\scriptsize $y_{\hspace{.8mm}8}$}
\put(6.09,.418){\tiny -}
\put(4.71,.25){\scriptsize $y_{\hspace{.8mm}4}$}
\put(4.86,.228){\tiny -}
\put(2.27,-.09){\scriptsize $y_0$}
\put(3.46,4.92){\scriptsize $y_3$}
\put(5.48,2.07){\scriptsize $y_7$}
\put(6.45,1.32){\scriptsize $y_{11}$}
\end{picture}
\caption{
A three-dimensional view of the phase portrait of \eqref{eq:bcnf} with \eqref{eq:param9p1}
(shown in Fig.~\ref{fig:qqExdRnonzero9p1})
using eigenvectors of $M_{\cX}$ ($\zeta_1$, $\zeta_2$ and $\zeta_3$) as coordinate axes.
\label{fig:qqExdRnonzero9p1_3d}
}
\end{center}
\end{figure}
%%%%%%%%%%%%%%%%%%%%%%%%%%%%%%%%%%%%%%%%%%%%%%%%%%%%%%%%%%%%%

Fig.~\ref{fig:qqExdRnonzero9p1_3d} provides a three-dimensional view of phase space about $x^{\cX}_0$
for this example.
Here we can observe several expected features of $W^u \!\left( \left\{ x^{\cX}_i \right\} \right)$.
The part of this manifold that converges to $x^{\cX}_0$ (with forward iterations)
is comprised of a union of line segments (connecting $y_3$, $y_7$, $y_{11}$, etc).
It lies within $W^s \!\left( \left\{ x^{\cX}_i \right\} \right)$
and approaches $x^{\cX}_0$ in the slow $\zeta_2$ direction.

Secondly we consider
\begin{equation}
\cX = RLR \;, \qquad
\cY = LL \;,
\label{eq:XY27}
\end{equation}
in order to show that $\cY$ need not consist of the last few symbols of $\cX$.
For \eqref{eq:XY27} we obtained the values
\begin{equation}
\begin{gathered}
\tau_L = -1.9465556255 \;, \qquad
\sigma_L = -1 \;, \qquad
\delta_L = 0.3387541740 \;, \\
\tau_R = -0.3249411658 \;, \qquad
\sigma_R = 1 \;, \qquad
\delta_R = 0.9 \;,
\end{gathered}
\label{eq:param27}
\end{equation}
and with these values a phase portrait of \eqref{eq:bcnf} is shown in Fig.~\ref{fig:qqExdRnonzero27}.
The words \eqref{eq:XY27} satisfy \eqref{eq:XY} with $\alpha = 4$, and so $y_4 \in \Sigma$.
Also $d = 1$, so part of the branch of $W^u \!\left( \left\{ x^{\cX}_i \right\} \right)$
that emanates from $x^{\cX}_0$ converges to $x^{\cX}_1$.
The other branch of $W^u \!\left( \left\{ x^{\cX}_i \right\} \right)$
converges to an asymptotically stable $LLR$-cycle.
Asymptotically stable $\cX^k \cY$-cycles and unstable $\cX^k \cY^{\overline{0}}$-cycles
are shown for $k = 2,\ldots,9$
(for $k = 0$ and $k = 1$ they are not admissible).

%%%%%%%%%%%%%%%%%%%%%%%%%%%%%%%%%%%%%%%%%%%%%%%%%%%%%%%%%%%%%
\begin{figure}[t!]
\begin{center}
\setlength{\unitlength}{1cm}
\begin{picture}(16,8)
\put(0,0){\includegraphics[width=16cm]{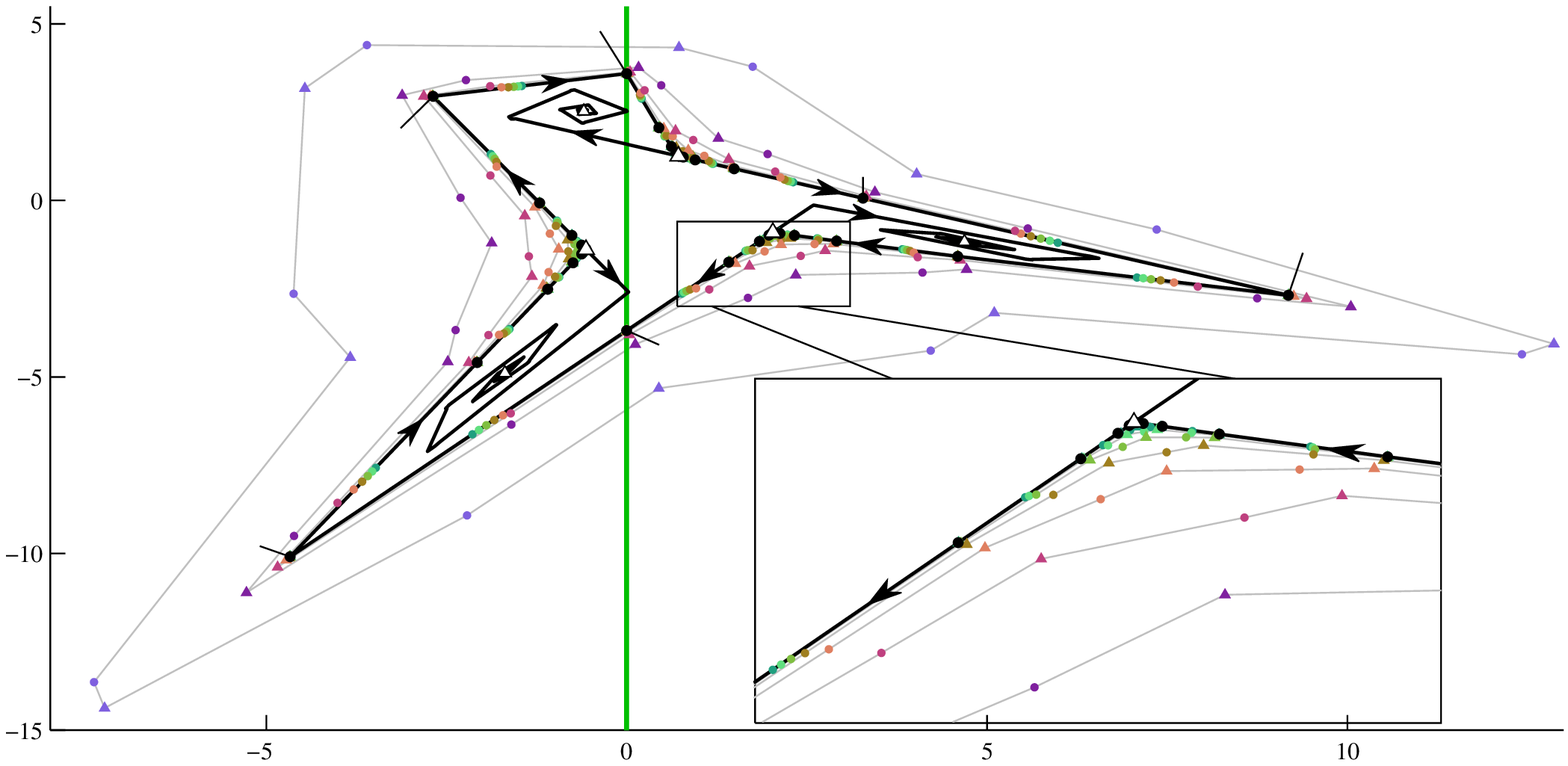}}
\put(8,0){\small $e_1^{\sf T} x$}
\put(0,4.27){\small $e_2^{\sf T} x$}
\put(11.27,3.98){\footnotesize $x^{\cX}_0$}
\put(6.3,5.76){\footnotesize $x^{\cX}_1$}
\put(6.97,6.1){\footnotesize $x^{\cX}_2$}
\put(6.46,2){\small $\Sigma$}
\put(9.73,2.89){\scriptsize $y_{\hspace{.8mm}3}$}
\put(9.88,2.868){\tiny -}
\put(5.98,6.13){\scriptsize $y_{\hspace{.8mm}2}$}
\put(6.13,6.108){\tiny -}
\put(8.93,6.4){\scriptsize $y_{\hspace{.8mm}1}$}
\put(9.08,6.378){\tiny -}
\put(7.11,4.59){\scriptsize $y_0$}
\put(4.4,6.64){\scriptsize $y_1$}
\put(13.3,5.67){\scriptsize $y_2$}
\put(2.87,2.7){\scriptsize $y_3$}
\put(6.35,7.83){\scriptsize $y_4$}
%\put(0,0){\scriptsize $y_{\hspace{.8mm}1}$}
%\put(.15,-.022){\tiny -}
\end{picture}
\caption{
A phase portrait of \eqref{eq:bcnf} with \eqref{eq:param27}
using the same conventions as Fig.~\ref{fig:qqExdRnonzero9p1}.
Here $\cX$ and $\cY$ are given by \eqref{eq:XY27}
and the $\cX^k \cY$ and $\cX^k \cY^{\overline{0}}$-cycles are shown for $k = 2,\ldots,9$.
\label{fig:qqExdRnonzero27}
}
\end{center}
\end{figure}
%%%%%%%%%%%%%%%%%%%%%%%%%%%%%%%%%%%%%%%%%%%%%%%%%%%%%%%%%%%%%

Finally we consider
\begin{equation}
\cX = RLRLRRLR \;, \qquad
\cY = LRRLR \;.
\label{eq:XY24}
\end{equation}
This example demonstrates that the scenario of Theorems \ref{th:infinity} and \ref{th:homoclinic}
is not restricted to only the simplest choices of $\cX$ and $\cY$.
For these words we obtained
\begin{equation}
\begin{gathered}
\tau_L = -0.5298581051 \;, \qquad
\sigma_L = 0.5 \;, \qquad
\delta_L = -0.2220122186 \;, \\
\tau_R = -3.4893057804 \;, \qquad
\sigma_R = 1.6 \;, \qquad
\delta_R = 0.6 \;,
\end{gathered}
\label{eq:param24}
\end{equation}
with which \eqref{eq:bcnf} is not invertible
because $\delta_L \delta_R < 0$.
Here $n = 8$ and $p = 5$ and so by \eqref{eq:d} we have $d = 3$.
Also $\alpha = 1$.
Fig.~\ref{fig:qqExdRnonzero24} shows a phase portrait of this example,
including asymptotically stable $\cX^k \cY$-cycles
and unstable $\cX^k \cY^{\overline{0}}$-cycles for $k = 0,\ldots,7$.

%%%%%%%%%%%%%%%%%%%%%%%%%%%%%%%%%%%%%%%%%%%%%%%%%%%%%%%%%%%%%
\begin{figure}[t!]
\begin{center}
\setlength{\unitlength}{1cm}
\begin{picture}(16,8)
\put(0,0){\includegraphics[width=16cm]{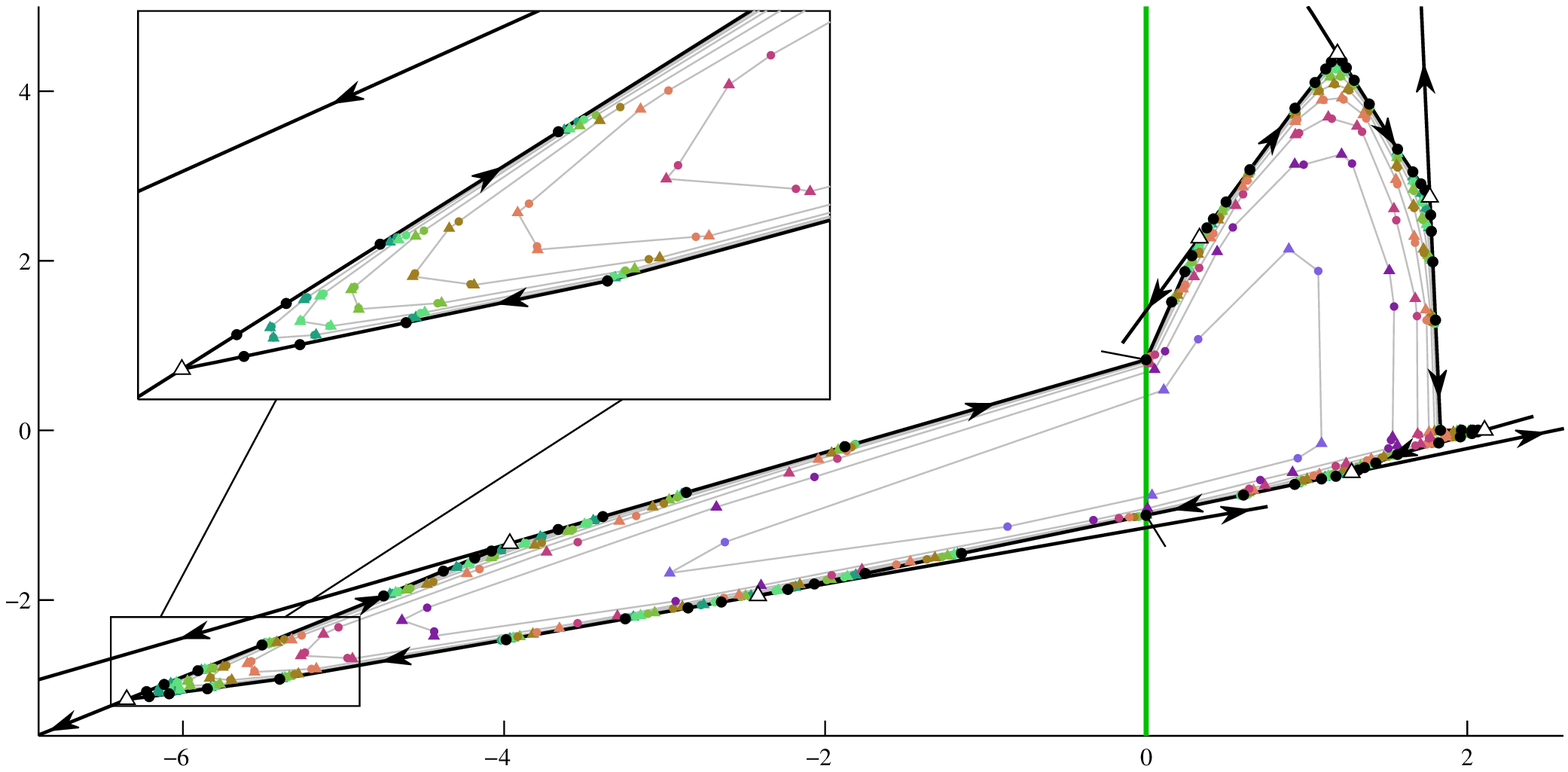}}
\put(8,0){\small $e_1^{\sf T} x$}
\put(0,4.55){\small $e_2^{\sf T} x$}
\put(13.7,3.06){\footnotesize $x^{\cX}_0$}
\put(5.53,2.96){\footnotesize $x^{\cX}_1$}
\put(14.74,6.15){\footnotesize $x^{\cX}_2$}
\put(7.86,1.83){\footnotesize $x^{\cX}_3$}
\put(11.98,5.83){\footnotesize $x^{\cX}_4$}
\put(15.03,4.06){\footnotesize $x^{\cX}_5$}
\put(2.13,4.6){\footnotesize $x^{\cX}_6$}
\put(13.8,7.58){\footnotesize $x^{\cX}_7$}
\put(12,2.58){\scriptsize $y_0$}
\put(11.15,4.61){\scriptsize $y_1$}
\end{picture}
\caption{
A phase portrait of \eqref{eq:bcnf} with \eqref{eq:param24}
using the same conventions as Fig.~\ref{fig:qqExdRnonzero9p1}.
Here $\cX$ and $\cY$ are given by \eqref{eq:XY24}
and the $\cX^k \cY$ and $\cX^k \cY^{\overline{0}}$-cycles are shown for $k = 0,\ldots,7$.
\label{fig:qqExdRnonzero24}
}
\end{center}
\end{figure}
%%%%%%%%%%%%%%%%%%%%%%%%%%%%%%%%%%%%%%%%%%%%%%%%%%%%%%%%%%%%%

%=====================================================================
\section{Proofs of the main results}
\label{sec:proofs}
\setcounter{equation}{0}

In order to prove Theorems \ref{th:infinity} and \ref{th:homoclinic},
we first introduce a coordinate system centred about $x^{\cX}_0$ and
aligned with the eigenvectors $\zeta_1$ and $\zeta_2$ and describe some basic properties of the map $g$ in these coordinates.
Here it is only necessary to assume that $\cX$ and $\cY$ are words of length $n$ and $p$ 
and that the eigenvalues of $M_{\cX}$ satisfy \eqref{eq:eigenvalueCondition}.

%We first introduce a coordinate system centred about $x^{\cX}_0$ and
%aligned with the eigenvectors $\zeta_1$ and $\zeta_2$.
Let $J$ be a matrix similar to $M_{\cX}$ and with the block diagonal form
\begin{equation}
J = \begin{bmatrix}
\lambda_1 \\
& \lambda_2 \\
& & \tilde{J}
\end{bmatrix},
\label{eq:J}
\end{equation}
where $\tilde{J}$ is an $(N-2) \times (N-2)$ matrix.
We could take $J$ to be a real Jordan form of $M_{\cX}$, but need not be so precise as for our purposes
we only require that large powers of $J$ can be written asymptotically as
\begin{equation}
J^k = \begin{bmatrix}
\lambda_1^k \\
& \lambda_2^k \\
& & \cO \!\left( \beta^k \right)
\end{bmatrix},
\label{eq:Jk}
\end{equation}
where $\cO$ represents big-O notation.
There exists a real-valued non-singular matrix $Q$ such that
\begin{equation}
J = Q^{-1} M_{\cX} Q \;,
%\label{eq:similarityTransform}
\nonumber
\end{equation}
and
\begin{equation}
\zeta_1 = Q e_1 \;, \qquad
\zeta_2 = Q e_2 \;.
\label{eq:zeta12}
\end{equation}
In view of \eqref{eq:eigenvectorNormalisation},
\begin{equation}
\omega_1^{\sf T} = e_1^{\sf T} Q^{-1} \;, \qquad
\omega_2^{\sf T} = e_2^{\sf T} Q^{-1} \;.
\label{eq:omega12}
\end{equation}
We then let
\begin{equation}
u = h(x) = Q^{-1} \!\left( x - x^{\cX}_0 \right).
\label{eq:h}
\nonumber
\end{equation}
These ``$u$-coordinates'' are centred about $x^{\cX}_0$ and partially align
with the invariant subspaces of $x^{\cX}_0$ for the map $f_{\cX}$.
In particular, $x \in E^u \!\left( x^{\cX}_0 \right)$
if and only if $h(x) = a e_1$ for some $a \in \mathbb{R}$,
and $x \in E^s \!\left( x^{\cX}_0 \right)$
if and only if $e_1^{\sf T} h(x) = 0$.

%We let $\tilde{\Sigma} = \left\{ u \in \mathbb{R}^N ~\middle|~
%e_1^{\sf T} \left( Q u + x^{\cX}_0 \right) = 0 \right\}$
%denote $\Sigma$ in $u$-coordinates.
We let $g = h \circ f \circ h^{-1}$
and $g_{\cZ} = h \circ f_{\cZ} \circ h^{-1}$
denote the transformations of $f$ and $f_{\cZ}$ to $u$-coordinates, for any word $\cZ$.
In particular,
\begin{equation}
g_{\cX}(u) = J u \;,
\label{eq:gX}
\end{equation}
and
\begin{equation}
g_{\cY}(u) = \Gamma u + \psi \;,
\label{eq:gY}
\end{equation}
where
\begin{equation}
\Gamma = Q^{-1} M_{\cY} Q \;, \qquad
\psi = Q^{-1} \left( P_{\cY} b - (I-M_{\cY}) x^{\cX}_0 \right).
\nonumber
\end{equation}
For each $i,j = 1,\ldots,N$, we let
$\gamma_{ij}$ denote the $(i,j)$-element of $\Gamma$ and $\psi_i = e_i^{\sf T} \psi$.
Then by \eqref{eq:zeta12} and \eqref{eq:omega12} the quantity \eqref{eq:c} is given by
\begin{equation}
c = \det \!\left(
\begin{bmatrix} e_1^{\sf T} \\ e_2^{\sf T} \end{bmatrix} \Gamma
\begin{bmatrix} e_1 & e_2 \end{bmatrix} \right) =
\gamma_{11} \gamma_{22} - \gamma_{12} \gamma_{21} \;.
\label{eq:c2}
\end{equation}
Given any $u \in \mathbb{R}^N$ and any word $\cZ$,
we say that $u$ follows $\cZ$ under $g$
if $h^{-1}(u)$ follows $\cZ$ under $f$.

We first provide the following algebraic result that is used below
in the proofs of both Theorem \ref{th:infinity} and Theorem \ref{th:homoclinic}.
This result is useful because in both proofs
$y_0 = h^{-1}(a_0 e_1)$ follows both $\cX \cY$ and $\cY \cX$ under $f$.

%...............................................................................
\begin{lemma}
If $a_0 \in \mathbb{R}$ is such that the point $a_0 e_1$ follows both $\cX \cY$ and $\cY \cX$ under $g$, then
\begin{equation}
\psi_2 = \frac{-a_0 (\lambda_1 - \lambda_2) \gamma_{21}}{1 - \lambda_2} \;.
\label{eq:psi2}
\end{equation}
\label{le:psi2}
\end{lemma}

%...............................................................................
\begin{proof}
Since $a_0 e_1$ follows $\cX \cY$ under $g$, we have
$g^{n+p}(a_0 e_1) = g_{\cY}(g_{\cX}(a_0 e_1))$.
Thus by \eqref{eq:gX} and \eqref{eq:gY},
\begin{equation}
g^{n+p}(a_0 e_1) = a_0 \Gamma J e_1 + \psi \;.
\label{eq:ynpp1}
\end{equation}
Similarly, since $a_0 e_1$ follows $\cY \cX$ under $g$, we have
$g^{n+p}(a_0 e_1) = g_{\cX}(g_{\cY}(a_0 e_1))$ and thus
\begin{equation}
g^{n+p}(a_0 e_1) = a_0 J \Gamma e_1 + J \psi \;.
\label{eq:ynpp2}
\end{equation}
By matching the second components of \eqref{eq:ynpp1} and \eqref{eq:ynpp2} we obtain
\begin{equation}
a_0 \lambda_2 \gamma_{21} + \lambda_2 \psi_2 = a_0 \lambda_1 \gamma_{21} + \psi_2 \;,
\nonumber
\end{equation}
where we have used \eqref{eq:J}.
By solving this equation for $\psi_2$ we obtain \eqref{eq:psi2}.
\end{proof}

%-------------------------------------------------------------------------------
\subsection{Proof of Theorem \ref{th:infinity}}
\label{sub:proof1}

%The theorem is proved in seven steps.
%We first describe the homoclinic $\cS$-orbit in $u$-coordinates (step 1),
%and argue that $\gamma_{11} = \psi_1 = 0$ because
%$f_{\cY} \!\left( E^u \!\left( x^{\cX}_0 \right) \right) \subset E^s \!\left( x^{\cX}_0 \right)$,
%as shown in Fig.~\ref{fig:InvManSchem} (step 2).
%We then use this to show that $\cX^k \cY$-cycles exist and are unique for sufficiently large values of $k$ (step 3),
%and obtain asymptotic expressions for the points of these periodic solutions (step 4).
%To prove $\cX^k \cY$-cycles are admissible we use separate arguments for points near $x^{\cX}_0$ (step 5)
%and points far from $x^{\cX}_0$ (step 6).
%Finally the stability of the $\cX^k \cY$-cycles is demonstrated with an additional transformation (step 7).

%^^^^^^^^^^^^^^^^^^^^^^^^^^^^^^^^^^^^^^^^^^^^^^^^^^^^^^^^^^^^^^^^^^^^^^^^^^^^^^^
\myStep{1}
Here we establish some important properties of the homoclinic $\cS$-orbit $\{ y_i \}$
and describe some of its points in $u$-coordinates.

By assumption, $y_0$ follows $\cY \cX^{\infty}$ under $f$.
As noted in the proof of Proposition \ref{pr:subsumed},
since $\cX \cY = \left( \cY \cX \right)^{\overline{0} \, \overline{\alpha}}$
and $y_0, y_{\alpha} \in \Sigma$,
by Lemma \ref{le:follow} $y_0$ also follows $\cX \cY \cX^{\infty}$ under $f$.
Thus for all $j \le 0$,
the point $y_{j n}$ follows $\cX$ under $f$.

Since $y_0 \in E^u \!\left( x^{\cX}_0 \right)$ by definition,
and $y_{j n} \to x^{\cX}_0$ as $j \to -\infty$\removableFootnote{
In order to argue that each $y_{j n} \in E^u \!\left( x^{\cX}_0 \right)$
it is not sufficient to note that $y_0 \in E^u \!\left( x^{\cX}_0 \right)$
and that each $y_{j n}$ follows $\cX$ because $f$ may not be invertible.
We need use the fact that $y_{j n} \to x^{\cX}_0$ as $j \to -\infty$.
},
we conclude that $y_{j n} \in E^u \!\left( x^{\cX}_0 \right)$ for all $j \le 1$.
In addition, we can write
\begin{equation}
h(y_{j n}) = a_j e_1 \;, \quad {\rm for~all~} j \le 1 \;,
\label{eq:hyjn}
\end{equation}
where
\begin{equation}
a_j = a_0 \lambda_1^j \;,
\label{eq:aj}
\end{equation}
and $a_0 \ne 0$ because $y_0 \in \Sigma$ and $x^{\cX}_0 \notin \Sigma$ and so $y_0 \ne x^{\cX}_0$\removableFootnote{
We could have $a_0 < 0$, but I'm pretty sure that throughout both proofs this is not a problem.
}.

For each $i = 0,\ldots,n-1$ and all $j \le -1$, we have
$y_{j n + i} \in \cL \!\left( x^{\cX}_i, y_i \right)$.
Since $y_0$ follows $\cX$ under $f$,
each $y_i$ (for $i = 0,\ldots,n-1$) follows $\cX_i$ under $f$.
Thus since the $\cS$-orbit and the $\cX$-cycle are admissible,
each $y_i$ either lies on $\Sigma$ or on the same side of $\Sigma$ as $x^{\cX}_i$.
Since each $x^{\cX}_i \notin \Sigma$ by assumption \eqref{it:infinityXcycle},
we have $\cL \!\left( x^{\cX}_i, y_i \right) \cap \Sigma = \varnothing$ and thus
\begin{equation}
y_{j n + i} \notin \Sigma \;, \quad {\rm for~each~}
i = 0,\ldots,n-1 {\rm ~and~all~} j \le -1 \;.
\label{eq:yjnpi}
\end{equation}

%^^^^^^^^^^^^^^^^^^^^^^^^^^^^^^^^^^^^^^^^^^^^^^^^^^^^^^^^^^^^^^^^^^^^^^^^^^^^^^^
\myStep{2}
Here we show that every $x \in E^u \!\left( x^{\cX}_0 \right)$ maps to
$E^s \!\left( x^{\cX}_0 \right)$ under $f_{\cY}$.
In terms of $g_{\cY}$, this equates to $\gamma_{11} = \psi_1 = 0$.

Since $y_0$ follows $\cY \cX^{\infty}$ under $f$ (by definition)
and $\{ y_i \}$ converges to the $\cX$-cycle as $i \to \infty$,
we must have $f_{\cY}(y_0) \in E^s \!\left( x^{\cX}_0 \right)$.
In Step 1 we showed that $y_0$ also follows $\cX \cY \cX^{\infty}$ under $f$,
thus $y_n$ follows $\cY \cX^{\infty}$ and so $f_{\cY}(y_n) \in E^s \!\left( x^{\cX}_0 \right)$.

In terms of $g_{\cY}$ we have
$e_1^{\sf T} g_{\cY}(a_0 e_1) = 0$ and $e_1^{\sf T} g_{\cY}(a_1 e_1) = 0$.
Since $a_0 \ne a_1$ and $e_1^{\sf T} g_{\cY}(a e_1) = \gamma_{11} a + \psi_1$, where $a \in \mathbb{R}$,
we must have
\begin{equation}
\gamma_{11} = 0 \;, \qquad
\psi_1 = 0 \;.
\label{eq:gamma11psi1}
\end{equation}
That is, every $x \in E^u \!\left( x^{\cX}_0 \right)$ maps to $E^s \!\left( x^{\cX}_0 \right)$ under $f_{\cY}$.
Also
\begin{equation}
c = -\gamma_{12} \gamma_{21} \;,
\label{eq:c3}
\end{equation}
by \eqref{eq:c2}.

%^^^^^^^^^^^^^^^^^^^^^^^^^^^^^^^^^^^^^^^^^^^^^^^^^^^^^^^^^^^^^^^^^^^^^^^^^^^^^^^
\myStep{3}
Next we use the previous result to show that
$\det \!\left( I - M_{\cX^k \cY} \right) \ne 0$
and thus $\cX^k \cY$-cycles exist and are unique, for sufficiently large values of $k$.

By \eqref{eq:MX}, $M_{\cX^k \cY} = M_{\cY} M_{\cX}^k$.
Also $\Gamma J^k = Q^{-1} M_{\cY} M_{\cX}^k Q$,
thus $\det \!\left( I - M_{\cX^k \cY} \right) = \det \!\left( I - \Gamma J^k \right)$.
In view of \eqref{eq:gamma11psi1} and the assumption $\lambda_1 \lambda_2 = 1$,
\begin{equation}
\Gamma J^k = \left[ \begin{array}{@{}c|c@{}}
\begin{array}{@{}cc@{}}
0 & \frac{\gamma_{12}}{\lambda_1^k} \\
\gamma_{21} \lambda_1^k & \frac{\gamma_{22}}{\lambda_1^k} \\
\vdots & \vdots \\
\gamma_{N1} \lambda_1^k & \frac{\gamma_{N2}}{\lambda_1^k}
\end{array} & \cO \!\left( \beta^k \right)
\end{array} \right],
\label{eq:GammaJk}
\end{equation}
and thus
\begin{equation}
\det \!\left( I - \Gamma J^k \right) = 1 + c -
\frac{\gamma_{22}}{\lambda_1^k} + \cO \!\left( \lambda_1^k \beta^k \right),
\label{eq:det}
\end{equation}
where we have also used \eqref{eq:c3}.
The leading order terms in \eqref{eq:det} stem from the top-left $2 \times 2$ block of $\Gamma J^k$.

Notice $\lambda_1 \beta < 1$ because $\beta < \lambda_2$ and $\lambda_1 \lambda_2 = 1$.
Thus $\det \!\left( I - \Gamma J^k \right) \to 1+c$ as $k \to \infty$.
Since $c \ne -1$ by assumption \eqref{it:infinitylam1lam2c}, we have $\det \!\left( I - \Gamma J^k \right) \ne 0$
and thus a unique $\cX^k \cY$-cycle, that we denote $\left\{ x^{\cX^k \cY}_i \right\}$,
for sufficiently large values of $k$.

%^^^^^^^^^^^^^^^^^^^^^^^^^^^^^^^^^^^^^^^^^^^^^^^^^^^^^^^^^^^^^^^^^^^^^^^^^^^^^^^
\myStep{4}
Here we derive asymptotic expressions for the points of the $\cX^k \cY$-cycles
associated with $x^{\cX}_0$ in $u$-coordinates.

The point $x^{\cX^k \cY}_0$ is a fixed point of $f_{\cY} \circ f_{\cX}^k$.
Thus $h \!\left( x^{\cX^k \cY}_0 \right)$ is a fixed point of $g_{\cY} \circ g_{\cX}^k$.
By \eqref{eq:gX} and \eqref{eq:gY},
\begin{equation}
h \!\left( x^{\cX^k \cY}_0 \right) = \left( I - \Gamma J^k \right)^{-1} \psi \;.
\label{eq:u0}
\end{equation}
In view of the definition of an adjugate matrix (as the transpose of a cofactor matrix),
we evaluate the minors of $I - \Gamma J^k$ using \eqref{eq:GammaJk} to produce
\begin{equation}
{\rm adj} \!\left( I - \Gamma J^k \right) \psi =
\begin{bmatrix} \frac{\gamma_{12} \psi_2}{\lambda_1^k} + \cO \!\left( \beta^k \right) \\
\psi_2 + \cO \!\left( \lambda_1^k \beta^k \right) \\
\cO(1) \\
\vdots \\
\cO(1) \end{bmatrix}.
\label{eq:adj}
\end{equation}
By then using \eqref{eq:det} and \eqref{eq:adj} to evaluate \eqref{eq:u0} we obtain
\begin{equation}
h \!\left( x^{\cX^k \cY}_0 \right) =
\begin{bmatrix} \frac{\gamma_{12} \psi_2}{(1+c) \lambda_1^k} +
\cO \!\left( \frac{1}{\lambda_1^{2 k}} \right) + \cO \!\left( \beta^k \right) \\
\frac{\psi_2}{1+c} + \cO \!\left( \frac{1}{\lambda_1^k} \right) +
\cO \!\left( \lambda_1^k \beta^k \right) \\
\cO(1) \\
\vdots \\
\cO(1) \end{bmatrix}.
\label{eq:u02}
\end{equation}
For all $j = 1,\ldots,k$, we have
$x^{\cX^k \cY}_{j n} = f_{\cX}^j \!\left( x^{\cX^k \cY}_0 \right)$.
Therefore, by \eqref{eq:gX},
\begin{equation}
h \!\left( x^{\cX^k \cY}_{j n} \right) = J^j h \!\left( x^{\cX^k \cY}_0 \right), \quad
{\rm for~all~} j = 0,\ldots,k \;.
\label{eq:ujn}
\end{equation}
In particular, by \eqref{eq:u02} and \eqref{eq:ujn} we can write
\begin{equation}
h \!\left( x^{\cX^k \cY}_{k n} \right) =
\frac{\gamma_{12} \psi_2}{1 + c} \,e_1 + \cO \!\left( \frac{1}{\lambda_1^k} \right) +
\cO \!\left( \lambda_1^k \beta^k \right).
\label{eq:ukn}
\end{equation}

%^^^^^^^^^^^^^^^^^^^^^^^^^^^^^^^^^^^^^^^^^^^^^^^^^^^^^^^^^^^^^^^^^^^^^^^^^^^^^^^
\myStep{5}
Next we identify a neighbourhood of $x^{\cX}_0$ within which all points follow $\cX$ under $f$.
As a first step to demonstrating the admissibility of the $\cX^k \cY$-cycles,
we use \eqref{eq:ujn} to show that,
for sufficiently large values of $k$,
all but a handful of the $x^{\cX^k \cY}_{j n}$ (which are dealt with in Step 6) lie within this neighbourhood.

By assumption \eqref{it:infinityXcycle}, the $\cX$-cycle is admissible with no points on $\Sigma$.
Thus there exists a neighbourhood of $x^{\cX}_0$ within which all points follow $\cX$ under $f$.
In $u$-coordinates this means that there exists $r > 0$ such that any point in the open ball
\begin{equation}
B_r = \left\{ u \in \mathbb{R}^N ~\middle|~ \| u \| < r \right\}
\label{eq:Br}
\end{equation}
follows $\cX$ under $g$.
By \eqref{eq:u02} and \eqref{eq:ujn}, there exists $j^* \in \mathbb{Z}$
such that for all $k > 2 j^*$ and all $j^* \le j < k-j^*$, we have
$h \!\left( x^{\cX^k \cY}_{j n} \right) \in B_r$, see Fig.~\ref{fig:uCoords}.
That is, $h \!\left( x^{\cX^k \cY}_{j n} \right)$ follows $\cX$ under $g$ for all $j^* \le j < k-j^*$.
More succinctly, $h \!\left( x^{\cX^k \cY}_{j^* n} \right)$ follows $\cX^{k - 2 j^*}$ under $g$.

%%%%%%%%%%%%%%%%%%%%%%%%%%%%%%%%%%%%%%%%%%%%%%%%%%%%%%%%%%%%%
\begin{figure}[b!]
\begin{center}
\setlength{\unitlength}{1cm}
\begin{picture}(16,8)
\put(0,0){\includegraphics[width=16cm]{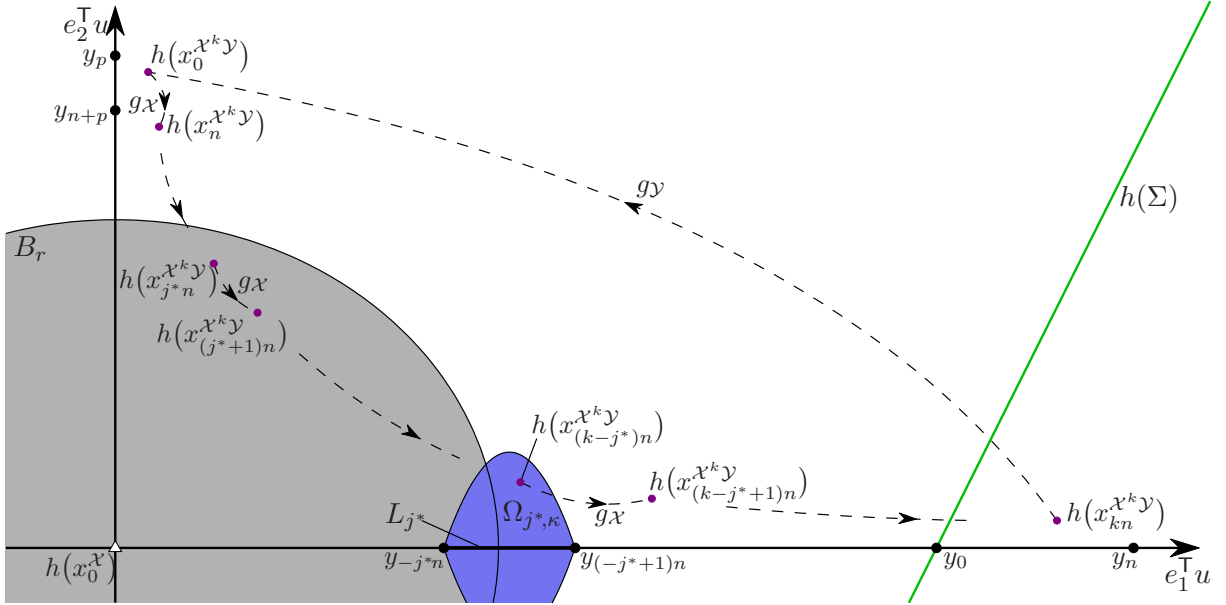}}
\put(15.4,.28){\small $e_1^{\sf T} u$}
\put(.76,7.6){\small $e_2^{\sf T} u$}
\put(14.8,5.3){\small $h(\Sigma)$}
\put(.1,4.63){\small $B_r$}
\put(6.6,1.1){\small $\Omega_{j^*\!,\kappa}$}
\put(5.04,1.08){\small $L_{j^*}$}
\put(1.02,7.23){\footnotesize $y_p$}
\put(.64,6.48){\footnotesize $y_{n+p}$}
\put(5.03,.53){\footnotesize $y_{-j^* \!n}$}
\put(7.63,.53){\footnotesize $y_{(-j^*+1) n}$}
\put(12.46,.53){\footnotesize $y_0$}
\put(14.7,.53){\footnotesize $y_n$}
\put(.52,.38){\footnotesize $h \big( x^{\cX}_0 \big)$}
\put(1.92,7.2){\footnotesize $h \big( x^{\cX^k \cY}_0 \big)$}
\put(2.13,6.27){\footnotesize $h \big( x^{\cX^k \cY}_n \big)$}
\put(1.5,4.2){\footnotesize $h \big( x^{\cX^k \cY}_{j^* n} \big)$}
\put(2,3.48){\footnotesize $h \big( x^{\cX^k \cY}_{(j^*+1)n} \big)$}
\put(6.99,2.27){\footnotesize $h \big( x^{\cX^k \cY}_{(k-j^*)n} \big)$}
\put(8.55,1.52){\footnotesize $h \big( x^{\cX^k \cY}_{(k-j^*+1)n} \big)$}
\put(14.1,1.1){\footnotesize $h \big( x^{\cX^k \cY}_{k n} \big)$}
\put(1.63,6.6){\footnotesize $g_{\cX}$}
\put(3.08,4.21){\footnotesize $g_{\cX}$}
\put(7.83,1.11){\footnotesize $g_{\cX}$}
\put(8.38,5.49){\footnotesize $g_{\cY}$}
\end{picture}
\caption{
A sketch of the phase space of $f$ in $u$-coordinates \eqref{eq:h}.
The first axis $e_1^{\sf T} u$ contains $E^u \big( x^{\cX}_0 \big)$.
The second axis $e_2^{\sf T} u$ corresponds to the slowest direction within $E^s \big( x^{\cX}_0 \big)$.
We show some points of the homoclinic $\cS$-orbit $\{ y_i \}$ and a typical $\cX^k \cY$-cycle.
%Each point within the ball $B_r$ follows $\cX$ under $g$.
%The line segment $L_{j^*}$ is defined by \eqref{eq:LjStar}
%and the region $\Omega_{j^*,\kappa}$ is defined by \eqref{eq:Omega}.
\label{fig:uCoords}
}
\end{center}
\end{figure}
%%%%%%%%%%%%%%%%%%%%%%%%%%%%%%%%%%%%%%%%%%%%%%%%%%%%%%%%%%%%%

%^^^^^^^^^^^^^^^^^^^^^^^^^^^^^^^^^^^^^^^^^^^^^^^^^^^^^^^^^^^^^^^^^^^^^^^^^^^^^^^
\myStep{6}
Here we complete the demonstration that $\cX^k \cY$-cycles are admissible for sufficiently large values of $k$.
In view of the final statement of Step 5, it remains to show that
$h \!\left( x^{\cX^k \cY}_{(k-j^*) n} \right)$ follows $\cX^{j^*} \cY \cX^{j^*}$ under $g$
for sufficiently large values of $k$.
We do this by showing that $h \!\left( x^{\cX^k \cY}_{(k-j^*) n} \right)$
approaches the line segment
\begin{equation}
L_{j^*} = \cL \!\left( h \!\left( y_{-j^* n} \right), h \!\left( y_{(-j^*+1)n} \right) \right),
\label{eq:LjStar}
\end{equation}
as $k \to \infty$, see Fig.~\ref{fig:uCoords}.

First, recall that $y_0$ follows $\cY \cX^{\infty}$ under $f$
and thus $y_{-j^* n}$ follows $\cX^{j^*} \cY \cX^{\infty}$ under $f$.
As shown in Step 1, $y_0$ follows $\cX \cY \cX^{\infty}$ under $f$,
thus $y_{(-j^*+1)n}$ also follows $\cX^{j^*} \cY \cX^{\infty}$ under $f$.
By Lemma \ref{le:lineSegment} every point in
$\cL \!\left( y_{-j^* n}, y_{(-j^*+1)n} \right)$ follows $\cX^{j^*} \cY \cX^{\infty}$ under $f$.
In terms of $u$-coordinates, this says that every $u \in L_{j^*}$ follows $\cX^{j^*} \cY \cX^{\infty}$ under $g$.

By assumption \eqref{it:infinityyi} and \eqref{eq:yjnpi}, the line segment
$f^i \!\left( \cL \!\left( y_{-j^* n}, y_{(-j^*+1)n} \right) \right)$
does not have both endpoints in $\Sigma$ for any $i \ge 0$.
Thus $f^i \!\left( \cL \!\left( y_{-j^* n}, y_{(-j^*+1)n} \right) \right) \cap \Sigma = \varnothing$,
for all $i \ge 0$.
Translated into $u$-coordinates, this says
\begin{equation}
g^i \!\left( L_{j^*} \right) \cap h(\Sigma) = \varnothing \;, \quad {\rm for~all~} i \ge 0 \;.
\label{eq:LjStarIntSigma}
\end{equation}
By \eqref{eq:hyjn}, we can write
\begin{equation}
L_{j^*} = \left\{ a e_1 ~\middle|~ a_{-j^*} < a < a_{-j^*+1} \right\}.
\label{eq:LjStar2}
\end{equation}
We now fatten $L_{j^*}$ into a rugby ball shaped region
\begin{equation}
\Omega_{j^*,\kappa} = \left\{ u \in \mathbb{R}^N ~\middle|~
a = e_1^{\sf T} u ,\,
a_{-j^*} < a < a_{-j^*+1} ,\,
\| u - a e_1 \| < \kappa \left( a - a_{-j^*} \right) \left( a_{-j^*+1} - a \right) \right\},
\label{eq:Omega}
\end{equation}
where $\kappa > 0$ is a measure of the magnitude of fattening, see Fig.~\ref{fig:uCoords}.
By \eqref{eq:LjStarIntSigma} there exists $\kappa > 0$ such that
$g^i \!\left( \Omega_{j^*,\kappa} \right) \cap h(\Sigma) = \varnothing$
for all $i = 0,\ldots,2 j^* n + p - 1$.
This is because each $g^i \!\left( L_{j^*} \right)$ is either bounded away from $h(\Sigma)$
or protrudes from $h(\Sigma)$ at a non-zero angle,
so we can find $\kappa > 0$ such that the finite collection of fattened sets
$\left\{ g^i \!\left( \Omega_{j^*,\kappa} \right) \right\}_{i=0}^{2 j^* n + p - 1}$
does not intersect $h(\Sigma)$.
Therefore every $u \in \Omega_{j^*,\kappa}$ follows $\cX^{j^*} \cY \cX^{j^*}$ under $g$.

We now fix $\kappa > 0$ at such a value.
It remains to show that
$h \!\left( x^{\cX^k \cY}_{(k-j^*)n} \right) \in \Omega_{j^*,\kappa}$
for sufficiently large values of $k$.
To do this we substitute \eqref{eq:psi2}
(which is valid because $a_0 e_1 = h(y_0)$ follows both $\cX \cY$ and $\cY \cX$ under $g$)
with $\lambda_1 \lambda_2 = 1$
into \eqref{eq:ukn} to obtain
\begin{equation}
h \!\left( x^{\cX^k \cY}_{k n} \right) =
\frac{a_0 (\lambda_1 + 1) c}{1+c} \,e_1 +
\cO \!\left( \frac{1}{\lambda_1^k} \right) +
\cO \!\left( \lambda_1^k \beta^k \right),
\label{eq:ukn2}
\end{equation}
where we have also used \eqref{eq:c3}.
Thus by \eqref{eq:aj} and \eqref{eq:ujn},
\begin{equation}
h \!\left( x^{\cX^k \cY}_{(k-j^*) n} \right) =
\frac{a_{-j^*} (\lambda_1 + 1) c}{1+c} \,e_1 +
\cO \!\left( \frac{1}{\lambda_1^k} \right) +
\cO \!\left( \lambda_1^k \beta^k \right).
\end{equation}
The assumption $\lambda_2 < c < 1$ implies\removableFootnote{
This is straightforward but requires several lines of working.
In fact we actually obtain the stronger result
\begin{equation}
1 < \frac{(\lambda_1 + 1) c}{1+c} < \frac{\lambda_1 + 1}{2} \;.
\end{equation}
which is interesting but not helpful for the proof.
}
\begin{equation}
1 < \frac{(\lambda_1 + 1) c}{1+c} < \lambda_1 \;,
\label{eq:admissibilityCondition}
\end{equation}
and thus
\begin{equation}
a_{-j^*} < \lim_{k \to \infty}
e_1^{\sf T} h \!\left( x^{\cX^k \cY}_{(k-j^*) n} \right) < a_{-j^*+1} \;.
\end{equation}
Hence $h \!\left( x^{\cX^k \cY}_{(k-j^*) n} \right) \in \Omega_{j^*,\kappa}$
and so the $\cX^k \cY$-cycle is admissible, for sufficiently large values of $k$.

%^^^^^^^^^^^^^^^^^^^^^^^^^^^^^^^^^^^^^^^^^^^^^^^^^^^^^^^^^^^^^^^^^^^^^^^^^^^^^^^
\myStep{7}
Finally we verify the asymptotic stability of the $\cX^k \cY$-cycles
by showing that all eigenvalues of $\Gamma J^k$,
and hence also of $M_{\cX^k \cY}$, have modulus less than $1$ for sufficiently large values of $k$.
To do this we apply the Gershgorin circle theorem to a matrix similar to $\Gamma J^k$.

By \eqref{eq:c3} and \eqref{eq:GammaJk}, it can be seen that $\Gamma J^k$ has an eigenvalue
that converges to the purely imaginary value ${\rm i} \sqrt{c}$ as $k \to \infty$.
The corresponding eigenvector is asymptotic to
\begin{equation}
v = \begin{bmatrix}
\frac{{\rm i} \sqrt{c}}{\lambda_1^k} \\
\gamma_{21} \\
\vdots \\
\gamma_{N1}
\end{bmatrix}.
\label{eq:eigenvector}
\end{equation}
This motivates use of the non-singular matrix
\begin{equation}
T = \left[ \begin{array}{@{}cc|@{\hspace{-.1mm}}c@{}}
v & \overline{v} &
\begin{array}{c} O \\ \hline I \end{array}
\end{array} \right],
\end{equation}
where $\overline{v}$ is the complex conjugate of $v$,
$O$ is a $2 \times (N-2)$ block of zeros,
and $I$ is the $(N-2) \times (N-2)$ identity matrix. 
Direct calculations produce
\begin{equation}
T^{-1} \Gamma J^k T = \left[ \begin{array}{@{}c|@{\hspace{-.1mm}}c@{}}
\begin{array}{cc}
{\rm i} \sqrt{c} + \cO \!\left( \lambda_1^k \beta^k \right) &
\cO \!\left( \lambda_1^k \beta^k \right) \\
\cO \!\left( \lambda_1^k \beta^k \right) &
-{\rm i} \sqrt{c} + \cO \!\left( \lambda_1^k \beta^k \right) \\
\cO \!\left( \frac{1}{\lambda_1^k} \right) & \cO \!\left( \frac{1}{\lambda_1^k} \right) \\
\vdots & \vdots \\
\cO \!\left( \frac{1}{\lambda_1^k} \right) & \cO \!\left( \frac{1}{\lambda_1^k} \right)
\end{array} &
\begin{array}{c}
\cO \!\left( \lambda_1^k \beta^k \right) \\
\cO \!\left( \lambda_1^k \beta^k \right) \\
\hline \\ \\
\cO \!\left( \beta^k \right) \\ \vspace{2mm} \\
\end{array}
\end{array} \right].
\end{equation}
By the Gershgorin circle theorem,
each eigenvalue of this matrix lies within
$\cO \!\left( \lambda_1^k \beta^k \right)$ of ${\rm i} \sqrt{c}$ or $-{\rm i} \sqrt{c}$
or within $\cO \!\left( \frac{1}{\lambda_1^k} \right)$ of $0$.
Since eigenvalues are invariant under similarity transforms,
the same is true for $\Gamma J^k$.
Since $\lambda_2 < c < 1$, by assumption,
all eigenvalues of $\Gamma J^k$ therefore have modulus less than $1$ for sufficiently large values of $k$.
%\end{proof}
\hfill$\blacksquare$

%-------------------------------------------------------------------------------
\subsection{Proof of Theorem \ref{th:homoclinic}}
\label{sub:proof2}

%This theorem is also proved in seven steps.
%We first obtain several restrictions on the coefficients describing $f$ (such as $\gamma_{11} = 0$)
%by appealing to the assumption that $\cX^k \cY$-cycles are stable (step 1).
%We then derive asymptotic expressions for points of the $\cX^k \cY$-cycles (step 2)
%and use the assumption of admissibility to obtain further restrictions on the coefficients
%(such as $\psi_1 = 0$ and $\lambda_1 \lambda_2 = 1$) (step 3).
%The fact that $\cX^k \cY$-cycles approach the $\cX$-cycle as $k \to \infty$
%allows us to show that the $\cX$-cycle is admissible (step 4).
%We then prove $y_{\alpha} \in \Sigma$ based on the identity \eqref{eq:XY} and the continuity of $f$ (step 5)
%and prove $c \le 1$ and $c \ge \lambda_2$ by admissibility and stability of $\cX^k \cY$-cycles, respectively (step 6).
%Finally we describe the $\cX^k \cY^{\overline{0}}$-cycles (step 7).

%^^^^^^^^^^^^^^^^^^^^^^^^^^^^^^^^^^^^^^^^^^^^^^^^^^^^^^^^^^^^^^^^^^^^^^^^^^^^^^^
\myStep{1}
First we use the assumption that $\cX^k \cY$-cycles are stable to obtain some
restrictions on the coefficients describing $f$, in particular $\gamma_{11} = 0$.

If an $\cX^k \cY$-cycle is stable
then the eigenvalues of $M_{\cX^k \cY}$ have modulus less than or equal to $1$.
This implies $\big| {\rm trace}(M_{\cX^k \cY}) \big| \le N$.
We have $g_{\cX^k \cY}(u) = \Gamma J^k u + \psi$, from which we obtain
\begin{equation}
{\rm trace}(M_{\cX^k \cY}) =
{\rm trace} \!\left( \Gamma J^k \right) =
\gamma_{11} \lambda_1^k + \gamma_{22} \lambda_2^k + \cO \!\left( \beta^k \right).
\label{eq:traceProof1}
\end{equation}
Since $\lambda_1 > 1$, we must have
\begin{equation}
\gamma_{11} = 0 \;,
\label{eq:gamma11}
\end{equation}
and therefore by \eqref{eq:c2}
\begin{equation}
c = -\gamma_{12} \gamma_{21} \;.
\label{eq:c3b}
\end{equation}
By using \eqref{eq:Jk} and substituting \eqref{eq:gamma11} and \eqref{eq:c3b} we obtain
\begin{equation}
\det \!\left( I - \Gamma J^k \right) =
1 + c \lambda_1^k \lambda_2^k - \gamma_{22} \lambda_2^k +
\cO \!\left( \lambda_1^k \beta^k \right).
\label{eq:det2}
\end{equation}
Again the assumption that $\cX^k \cY$-cycles are stable implies
$\det \!\left( I - \Gamma J^k \right) \not\to \infty$ as $k \to \infty$.
Since $c \ne 0$ by assumption, we must have
\begin{equation}
\lambda_1 \lambda_2 \le 1 \;.
\label{eq:lambda12Proof1}
\end{equation}
If $\lambda_1 \lambda_2 = 1$, then $\det \!\left( I - \Gamma J^k \right) \to 1 + c$ as $k \to \infty$,
while if $0 < \lambda_1 \lambda_2 < 1$ (which turns out not to be the case),
then $\det \!\left( I - \Gamma J^k \right) \to 1$ as $k \to \infty$.
In either case the limiting value is nonzero because $c \ne -1$ by assumption.

%^^^^^^^^^^^^^^^^^^^^^^^^^^^^^^^^^^^^^^^^^^^^^^^^^^^^^^^^^^^^^^^^^^^^^^^^^^^^^^^
\myStep{2}
Here we derive asymptotic expressions for points of the $\cX^k \cY$-cycles associated with $x^{\cX}_0$
in $u$-coordinates.

We have $h \!\left( x^{\cX^k \cY}_0 \right) = \left( I - \Gamma J^k \right)^{-1} \psi$
where $\psi$ is not the zero vector (for otherwise
$h \!\left( x^{\cX^k \cY}_0 \right) = h \!\left( x^{\cX}_0 \right)$ for all $k$, which is not possible).
In addition, $h \!\left( x^{\cX^k \cY}_{j n} \right) = J^j h \!\left( x^{\cX^k \cY}_0 \right)$
for all $j = 0,\ldots,k$.
By evaluating the minors of $I - \Gamma J^k$ subject to \eqref{eq:gamma11}
(in a similar fashion to Step 4 of the proof of Theorem \ref{th:infinity})\removableFootnote{
I derived
\begin{equation}
%\hspace{-15mm}
{\rm adj} \!\left( I - \Gamma J^k \right) =
\left[ \begin{array}{ccccc}
1 - \gamma_{22} \lambda_2^k + \cO \!\left( \beta^k \right) &
\gamma_{12} \lambda_2^k + \cO \!\left( \lambda_2^k \beta^k \right) &
\cO \!\left( \beta^k \right) & \cdots & \cO \!\left( \beta^k \right) \\
\gamma_{21} \lambda_1^k + \cO \!\left( \lambda_1^k \beta^k \right) &
1 + \cO \!\left( \lambda_1^k \beta^k \right) &
\cO \!\left( \lambda_1^k \beta^k \right) & \cdots & \cO \!\left( \lambda_1^k \beta^k \right) \\
\gamma_{31} \lambda_1^k + \cO \!\left( \lambda_1^k \lambda_2^k \right) &
\cO \!\left( \lambda_1^k \lambda_2^k \right) &
1 + \cO \!\left( \lambda_1^k \lambda_2^k \right) & & \vdots \\
\vdots & \vdots & & \ddots \\
\gamma_{N1} \lambda_1^k + \cO \!\left( \lambda_1^k \lambda_2^k \right) &
\cO \!\left( \lambda_1^k \lambda_2^k \right) & \cdots & &
1 + \cO \!\left( \lambda_1^k \lambda_2^k \right)
\end{array} \right],
\end{equation}
by explicitly computing the cofactor matrix of $I - \Gamma J^k$.
This enabled me to deduce the error terms.
},
we obtain
\begin{equation}
h \!\left( x^{\cX^k \cY}_{j n} \right) = \frac{1}{\det \!\left( I - \Gamma J^k \right)}
\left[ \begin{array}{c}
\lambda_1^j \left( 1 - \gamma_{22} \lambda_2^k \right) \psi_1 +
\gamma_{12} \lambda_1^j \lambda_2^k \psi_2 +
\cO \!\left( \lambda_1^j \beta^k \right) \\
\gamma_{21} \lambda_1^k \lambda_2^j \psi_1 +
\lambda_2^j \psi_2 +
\cO \!\left( \lambda_1^k \lambda_2^j \beta^k \right) \\
\cO \!\left( \lambda_1^k \beta^j \right) \\
\vdots \\
\cO \!\left( \lambda_1^k \beta^j \right) \\
\end{array} \right], \quad
{\rm for~all~} j = 0,\ldots,k \;.
\label{eq:uSjnX2}
\end{equation}

%^^^^^^^^^^^^^^^^^^^^^^^^^^^^^^^^^^^^^^^^^^^^^^^^^^^^^^^^^^^^^^^^^^^^^^^^^^^^^^^
\myStep{3}
Here we use the admissibility of the $\cX^k \cY$-cycles to obtain further restrictions on the coefficients
describing $f$, in particular $\psi_1 = 0$ and $\lambda_1 \lambda_2 = 1$.

For each value of $k$ for which the $\cX^k \cY$-cycle is admissible with no points on $\Sigma$,
the points $x^{\cX^k \cY}_{(k-1)n}$ and $x^{\cX^k \cY}_{k n}$ lie on different sides of $\Sigma$ because
$\cX_0 \ne \cY_0$.
Thus $h \!\left( x^{\cX^k \cY}_{(k-1)n} \right)$ and $h \!\left( x^{\cX^k \cY}_{k n} \right)$
lie on different sides of $h(\Sigma)$ for arbitrarily large values of $k$.

But if $\psi_1 \ne 0$, then by \eqref{eq:uSjnX2} as $k \to \infty$ the first components of 
$h \!\left( x^{\cX^k \cY}_{(k-1)n} \right)$ and $h \!\left( x^{\cX^k \cY}_{k n} \right)$
diverge while all other components converge.
Thus $h(\Sigma)$ must be parallel to the first coordinate axis.
In $x$-coordinates this says that $\zeta_1$ is parallel to $\Sigma$.
But this is not possible because $e_1^{\sf T} \zeta_1 \ne 0$, by assumption, thus
\begin{equation}
\psi_1 = 0 \;.
\label{eq:psi1}
\end{equation}
Therefore \eqref{eq:uSjnX2} reduces to
\begin{equation}
h \!\left( x^{\cX^k \cY}_{j n} \right) = \frac{1}{\det \!\left( I - \Gamma J^k \right)}
\left[ \begin{array}{c}
\gamma_{12} \lambda_1^j \lambda_2^k \psi_2 +
\cO \!\left( \lambda_1^j \beta^k \right) \\
\lambda_2^j \psi_2 +
\cO \!\left( \lambda_1^k \lambda_2^j \beta^k \right) \\
\cO \!\left( \beta^j \right) \\
\vdots \\
\cO \!\left( \beta^j \right)
\end{array} \right], \quad
{\rm for~all~} j = 0,\ldots,k \;.
\label{eq:uSjnX3}
\end{equation}

Notice $\gamma_{12} \ne 0$, by \eqref{eq:c3b} and because $c \ne 0$ by assumption.
Now suppose for a contradiction that either $\lambda_1 \lambda_2 < 1$ or $\psi_2 = 0$.
In this case $h \!\left( x^{\cX^k \cY}_{(k-1) n} \right)$ and $h \!\left( x^{\cX^k \cY}_{k n} \right)$
converge to the origin as $k \to \infty$.
Since $f$ is continuous, this implies that the distance between the
$(n+p)^{\rm th}$ iterates of
$h \!\left( x^{\cX^k \cY}_{(k-1) n} \right)$ and $h \!\left( x^{\cX^k \cY}_{k n} \right)$
tends to zero as $k \to \infty$.
However, these iterates are the points
$h \!\left( x^{\cX^k \cY}_0 \right)$ and $h \!\left( x^{\cX^k \cY}_n \right)$.
We have $h \!\left( x^{\cX^k \cY}_0 \right) = g_{\cY} \!\left( h \!\left( x^{\cX^k \cY}_{k n} \right) \right)$
which by \eqref{eq:gY} converges to $\psi$ as $k \to \infty$
(given that $h \!\left( x^{\cX^k \cY}_{k n} \right)$ converges to the origin).
Thus by \eqref{eq:gX}, $h \!\left( x^{\cX^k \cY}_{k n} \right) \to J \psi$ as $k \to \infty$.
But $\psi \ne J \psi$, for otherwise $J$ would have an eigenvalue $1$ which is not possible
by \eqref{eq:eigenvalueCondition}.
Thus the distance between $h \!\left( x^{\cX^k \cY}_0 \right)$ and $h \!\left( x^{\cX^k \cY}_n \right)$
does not go to zero as $k \to \infty$.
This a contradiction, hence $\lambda_1 \lambda_2 = 1$ and $\psi_2 \ne 0$.

%^^^^^^^^^^^^^^^^^^^^^^^^^^^^^^^^^^^^^^^^^^^^^^^^^^^^^^^^^^^^^^^^^^^^^^^^^^^^^^^
\myStep{4}
Here we use the observation that points of $\cX^k \cY$-cycles converge to those of the $\cX$-cycle
as $k \to \infty$ to show that the $\cX$-cycle is admissible and $x^{\cX}_0 \notin \Sigma$.

By \eqref{eq:uSjnX3} with $j = \left\lfloor \frac{k}{2} \right\rfloor$,
the point $h \!\left( x^{\cX^k \cY}_{j n} \right)$ converges to the origin as $k \to \infty$.
In $x$-coordinates this means that $x^{\cX^k \cY}_{j n} \to x^{\cX}_0$ as $k \to \infty$.
Each $x^{\cX^k \cY}_{j n}$ follows $\cX$ under $f$ for sufficiently large values of $k$
(by the admissibility of the $\cX^k \cY$-cycles).
Thus $x^{\cX}_0$ must also follow $\cX$ under $f$ because $f$ is continuous.
That is, the $\cX$-cycle is admissible.

By \eqref{eq:det2} and \eqref{eq:uSjnX3} with $\lambda_1 \lambda_2 = 1$,
\begin{equation}
h \!\left( x^{\cX^k \cY}_{(k-1)n} \right) \to \frac{\gamma_{12} \psi_2}{(1+c) \lambda_1} \,e_1 \;, \quad
h \!\left( x^{\cX^k \cY}_{k n} \right) \to \frac{\gamma_{12} \psi_2}{1+c} \,e_1 \;,
\nonumber
\end{equation}
as $k \to \infty$.
But $h \!\left( x^{\cX^k \cY}_{(k-1)n} \right)$ and $h \!\left( x^{\cX^k \cY}_{k n} \right)$ lie on different
sides of $h(\Sigma)$ for arbitrarily large values of $k$,
so if $h(\Sigma)$ passes through the origin (so that $x^{\cX}_0 \in \Sigma$)
it must contain the first coordinate axis.
As in Step 3, this is not possible because $e_1^{\sf T} \zeta_1 \ne 0$.
Therefore $x^{\cX}_0 \notin \Sigma$.

%^^^^^^^^^^^^^^^^^^^^^^^^^^^^^^^^^^^^^^^^^^^^^^^^^^^^^^^^^^^^^^^^^^^^^^^^^^^^^^^
\myStep{5}
Here we construct the homoclinic $\cS$-orbit $\{ y_i \}$.
Our main efforts are in showing that $y_{\alpha} \in \Sigma$.
For this step we work in $x$-coordinates because our arguments utilise the continuity of $f$
and the formula \eqref{eq:xi}.

Since $\gamma_{11} = 0$ and $\psi_1 = 0$,
by \eqref{eq:gY} we have $e_1^{\sf T} g_{\cY}(u) = 0$ for any $u = a e_1$, where $a \in \mathbb{R}$.
In $x$-coordinates this means that every $x \in E^u \!\left( x^{\cX}_0 \right)$
maps to $E^s \!\left( x^{\cX}_0 \right)$ under $f_{\cY}$.

Let $y_0 = E^u \!\left( x^{\cX}_0 \right) \cap \Sigma$,
which is a unique point because $e_1^{\sf T} \zeta_1 \ne 0$ by assumption.
Since $y_0 \in E^u \!\left( x^{\cX}_0 \right)$, we have
$f_{\cY}(y_0) \in E^s \!\left( x^{\cX}_0 \right)$.
Therefore there exists an $\cS$-orbit $\{ y_i \}$
involving $y_0$ that is homoclinic to the $\cX$-cycle (although it may not be admissible).
Also $y_0 \in \Sigma$, thus by Lemma \ref{le:follow},
$f_{\cY^{\overline{0}}}(y_0) = f_{\cY}(y_0)$ and so
$f_{\cY^{\overline{0}}}(y_0) \in E^s \!\left( x^{\cX}_0 \right)$.
Then by assumption \eqref{it:homoclinicY0bar}, $y_0$ must be the only element of $E^u \!\left( x^{\cX}_0 \right)$
that maps to $E^s \!\left( x^{\cX}_0 \right)$ under $f_{\cY^{\overline{0}}}$.

Let $\tilde{\cX}$ denote the first $\alpha$ elements of $\cX \cY$
and $\tilde{\cY}$ denote the remaining elements of $\cX \cY$.
That is
\begin{equation}
\tilde{\cX} \tilde{\cY} = \cX \cY \;,
\label{eq:XtildeYtilde}
\end{equation}
where $\tilde{\cX}$ is of length $\alpha$
and $\tilde{\cY}$ is of length $n+p-\alpha$.
Since $\left( \cX \cY \right)^{\overline{0} \, \overline{\alpha}} = \cY \cX$,
\begin{equation}
\tilde{\cX}^{\overline{0}} \tilde{\cY}^{\overline{0}} = \cY \cX \;.
\label{eq:Xtilde0barYtilde0bar}
\end{equation}
Since $y_0 \in \Sigma$,
\begin{equation}
y_{\alpha} = f_{\tilde{\cX}}(y_0) = f_{\tilde{\cX}^{\overline{0}}}(y_0) \;,
\label{eq:yalpha}
\end{equation}
by Lemma \ref{le:follow}.

For any $a \in \mathbb{R}$, we have $x^{\cX}_0 + a \zeta_1 \in E^u \!\left( x^{\cX}_0 \right)$.
Consider the affine function
\begin{equation}
\phi_1(a) = e_1^{\sf T} f_{\tilde{\cX}} \!\left( x^{\cX}_0 + a \zeta_1 \right).
\label{eq:phi1}
\end{equation}
Below we show that $\frac{d \phi_1}{d a} \ne 0$.
Given that this is true, there exists unique $\hat{a} \in \mathbb{R}$ such that $\phi_1(\hat{a}) = 0$.
Let $\hat{x} = x^{\cX}_0 + \hat{a} \zeta_1$.
Then $f_{\tilde{\cX}}(\hat{x}) \in \Sigma$, and so
\begin{equation}
f_{\tilde{\cX} \tilde{\cY}}(\hat{x}) = f_{\left( \tilde{\cX} \tilde{\cY} \right)^{\overline{\alpha}}}(\hat{x}) \;,
\label{eq:xHatIdentity1}
\end{equation}
by Lemma \ref{le:follow}.
By $\cX \cY = \left( \cY \cX \right)^{\overline{0} \, \overline{\alpha}}$ and \eqref{eq:XtildeYtilde},
we can alter the symbols in \eqref{eq:xHatIdentity1} to produce
\begin{equation}
f_{\cX \cY}(\hat{x}) = f_{\cY^{\overline{0}} \cX}(\hat{x}) \;.
\label{eq:xHatIdentity2}
\end{equation}
Since $\hat{x} \in E^u \!\left( x^{\cX}_0 \right)$, we have
$f_{\cX}(\hat{x}) \in E^u \!\left( x^{\cX}_0 \right)$
(because $E^u \!\left( x^{\cX}_0 \right)$ is invariant under $f_{\cX}$),
and thus $f_{\cX \cY}(\hat{x}) \in E^s \!\left( x^{\cX}_0 \right)$.
That is, $f_{\cY^{\overline{0}} \cX}(\hat{x}) \in E^s \!\left( x^{\cX}_0 \right)$, by \eqref{eq:xHatIdentity2}.
Thus $f_{\cY^{\overline{0}}}(\hat{x}) \in E^s \!\left( x^{\cX}_0 \right)$
(because $E^s \!\left( x^{\cX}_0 \right)$ is invariant under $f_{\cX}$).
Above we showed that $y_0$ is the only element of $E^u \!\left( x^{\cX}_0 \right)$
that maps to $E^s \!\left( x^{\cX}_0 \right)$ under $f_{\cY^{\overline{0}}}$, hence $\hat{x} = y_0$.
Thus $f_{\tilde{\cX}}(\hat{x})$ is the point $y_{\alpha}$, and so $y_{\alpha} \in \Sigma$.

To complete this step we show that $\frac{d \phi_1}{d a} \ne 0$.
Consider the affine function
\begin{equation}
\phi_2(a) = \omega_1^{\sf T} \!\left( f_{\cY^{\overline{0}} \cX}
\!\left( x^{\cX}_0 + a \zeta_1 \right) - x^{\cX}_0 \right) \;.
\label{eq:phi2}
\end{equation}
This function returns a measure of the distance of the point
$f_{\cY^{\overline{0}} \cX} \!\left( x^{\cX}_0 + a \zeta_1 \right)$ from $E^s \!\left( x^{\cX}_0 \right)$,
where $x^{\cX}_0 + a \zeta_1 \in E^u \!\left( x^{\cX}_0 \right)$.
We have $\phi_2(a_0) = 0$ because $h(y_0) = a_0 e_1$ and so $y_0 = x^{\cX}_0 + a_0 \zeta_1$.
Since $y_0$ is the only element of $E^u \!\left( x^{\cX}_0 \right)$ that
maps to $E^s \!\left( x^{\cX}_0 \right)$ under $f_{\cY^{\overline{0}}}$,
and $E^s \!\left( x^{\cX}_0 \right)$ is invariant under $f_{\cX}$,
the value $a_0$ is the unique zero of $\phi_2(a)$.
Thus we must have $\frac{d \phi_2}{d a} \ne 0$.

From \eqref{eq:xi} and \eqref{eq:MX} we obtain the formula
\begin{equation}
f_{\cY^{\overline{0}} \cX}(x) =
f_{\left( \cY \cX \right)^{\overline{0} \, \overline{\alpha}}}(x) +
\xi_{\tilde{\cY}} e_1^{\sf T} f_{\tilde{\cX}}(x) \;,
\label{eq:fY0barX}
\end{equation}
where
\begin{equation}
\xi_{\tilde{\cY}} = A_{\tilde{\cY}_{n+p-\alpha-1}} \cdots A_{\tilde{\cY}_1} \xi \;.
\end{equation}
By using $\cX \cY = \left( \cY \cX \right)^{\overline{0} \, \overline{\alpha}}$,
we can therefore write
\begin{equation}
\phi_2(a) = \omega_1^{\sf T} \!\left(
f_{\cX \cY} \!\left( x^{\cX}_0 + a \zeta_1 \right) - x^{\cX}_0 + \xi_{\tilde{\cY}} \phi_1(a) \right).
\nonumber
\end{equation}
But $\omega_1^{\sf T} \!\left( f_{\cX \cY} \!\left( x^{\cX}_0 + a \zeta_1 \right) - x^{\cX}_0 \right)$
is identically zero because every $x \in E^u \!\left( x^{\cX}_0 \right)$
maps to $E^s \!\left( x^{\cX}_0 \right)$ under $f_{\cY}$, thus
$\phi_2(a) = \omega_1^{\sf T} \xi_{\tilde{\cY}} \phi_1(a)$.
Therefore $\frac{d \phi_1}{d a} \ne 0$ because $\frac{d \phi_2}{d a} \ne 0$.

%^^^^^^^^^^^^^^^^^^^^^^^^^^^^^^^^^^^^^^^^^^^^^^^^^^^^^^^^^^^^^^^^^^^^^^^^^^^^^^^
\myStep{6}
Next we use the assumption that $\cX^k \cY$-cycles are admissible to show that $c \ge \lambda_2$,
and use the assumption that $\cX^k \cY$-cycles are stable to show that $c \le 1$.

Equation \eqref{eq:GammaJk} is a formula for $\Gamma J^k$ obtained in the proof of Theorem \ref{th:infinity}.
This formula is valid here
because we have shown that $\gamma_{11} = 0$, $\psi_1 = 0$ and $\lambda_1 \lambda_2 = 1$.
As noted above, it can be seen from \eqref{eq:GammaJk} that $\Gamma J^k$ has an
eigenvalue that converges to ${\rm i} \sqrt{c}$ as $k \to \infty$.
The stability of the $\cX^k \cY$-cycles requires all eigenvalues of $\Gamma J^k$
to have modulus less than or equal to $1$, hence we must have $|c| \le 1$.

Equation \eqref{eq:uSjnX3} with $j = k$ gives
\begin{equation}
h \!\left( x^{\cX^k \cY}_{k n} \right) = \frac{\gamma_{12} \psi_2}{1+c} \,e_1 +
\cO \!\left( \frac{1}{\lambda_1^k} \right) + \cO \!\left( \lambda_1^k \beta^k \right),
\nonumber
\end{equation}
where we have substituted \eqref{eq:det2} and $\lambda_1 \lambda_2 = 1$
(and recall $c \ne -1$, by assumption).
By then substituting \eqref{eq:psi2} 
(since we have shown that $a_0 e_1 = h(y_0)$ follows both $\cX \cY$ and $\cY \cX$ under $g$)
with $\lambda_1 \lambda_2 = 1$ and \eqref{eq:c3b}, we obtain
\begin{equation}
h \!\left( x^{\cX^k \cY}_{k n} \right) =
\frac{a_0 (\lambda_1 + 1) c}{1+c} \,e_1 +
\cO \!\left( \frac{1}{\lambda_1^k} \right) + \cO \!\left( \lambda_1^k \beta^k \right).
\nonumber
\end{equation}
By the admissibility of the $\cX^k \cY$-cycles,
and the admissibility of the $\cX$-cycle (see Step 4),
we know that the points $x^{\cX}_0$ and $x^{\cX^k \cY}_{k n}$ lie on different sides of $\Sigma$.
Since $h \!\left( x^{\cX}_0 \right)$ is the origin and $a_0 e_1 \in h(\Sigma)$,
we must have $\frac{(\lambda_1 + 1) c}{1+c} \ge 1$.
This implies $c \ge \lambda_2$ (using $\lambda_1 \lambda_2 = 1$ and $c > -1$).

%^^^^^^^^^^^^^^^^^^^^^^^^^^^^^^^^^^^^^^^^^^^^^^^^^^^^^^^^^^^^^^^^^^^^^^^^^^^^^^^
\myStep{7}
Finally we establish the properties of the $\cX^k \cY^{\overline{0}}$-cycles.

We write the map $g_{\cY^{\overline{0}}}$ as
\begin{equation}
g_{\cY^{\overline{0}}}(u) = \Xi u + \chi \;,
\label{eq:gY0bar}
\end{equation}
where $\Xi = Q^{-1} M_{\cY^{\overline{0}}} Q$
and $\chi = Q^{-1} \left( P_{\cY^{\overline{0}}} b - \left( I - M_{\cY^{\overline{0}}} \right) x^{\cX}_0 \right)$.
For each $i,j = 1,\ldots,N$, we let
$\xi_{i j}$ denote the $(i,j)$-element of $\Xi$ and $\chi_i = e_i^{\sf T} \chi$.
Assumption \eqref{it:homoclinicY0bar}, which states that 
$f_{\cY^{\overline{0}}}$ does not map $E^u \!\left( x^{\cX}_0 \right)$ to $E^s \!\left( x^{\cX}_0 \right)$,
implies that $e_1^{\sf T} g_{\cY^{\overline{0}}}(a e_1)$ is nonzero for some $a \in \mathbb{R}$.
But $e_1^{\sf T} g_{\cY^{\overline{0}}}(a_0 e_1) = 0$
because $y_0 = h^{-1}(a_0 e_1)$ maps to $E^s \!\left( x^{\cX}_0 \right)$ under $f_{\cY}$ and $y_0 \in \Sigma$.
By \eqref{eq:gY0bar}, we have $e_1^{\sf T} g_{\cY^{\overline{0}}}(a e_1) = \xi_{11} a + \chi_1$.
Thus we must have $\xi_{11} \ne 0$ and
\begin{equation}
\chi_1 = -a_0 \xi_{11} \;.
\label{eq:chi1}
\end{equation}
By \eqref{eq:Jk} we obtain
\begin{equation}
\det \!\left( I - \Xi J^k \right) = -\xi_{11} \lambda_1^k + \cO(1) \;,
\label{eq:det3}
\end{equation}
which is nonzero for sufficiently large values of $k$.
Therefore $\cX^k \cY^{\overline{0}}$-cycles exist and are unique (although may not be admissible)
for sufficiently large values of $k$, and we denote them $\left\{ x^{\cX^k \cY^{\overline{0}}}_i \right\}$.

The point $h \!\left( x^{\cX^k \cY^{\overline{0}}}_{k n} \right)$ is a fixed point of
$g_{\cY^{\overline{0}} \cX^k}(u) = J^k \left( \Xi u + \chi \right)$.
By using \eqref{eq:Jk} and \eqref{eq:chi1} to solve this fixed point equation, we obtain
\begin{equation}
h \!\left( x^{\cX^k \cY^{\overline{0}}}_{k n} \right) =
a_0 e_1 + \cO \!\left( \frac{1}{\lambda_1^k} \right),
\end{equation}
and hence $x^{\cX^k \cY^{\overline{0}}}_{k n} \to y_0$ as $k \to \infty$.
%\end{proof}
\hfill$\blacksquare$

%=====================================================================
\section{Discussion}
\label{sec:conc}
\setcounter{equation}{0}

This paper provides rigorous results for piecewise-linear continuous maps \eqref{eq:f} of any number of dimensions.
From a bifurcation theory perspective, results of such generality are relatively rare because
one cannot usually reduce the dimensionality of \eqref{eq:f} through a centre manifold analysis.
Other rigorous results for \eqref{eq:f} with an arbitrary number of dimensions include
the border-collision normal form \cite{Di03},
the characterisation of fixed points and period-two solutions \cite{Fe78,Si14d},
invariant measures and $N$-dimensional attractors \cite{Gl14,Gl15b},
and shrinking points of mode-locking regions \cite{SiMe09,Si15c}.

Theorems \ref{th:infinity} and \ref{th:homoclinic} provide an equivalence between
subsumed homoclinic connections and infinitely many asymptotically stable periodic solutions in a codimension-three scenario.
We have chosen to state the two results separately as they provide considerably more information
than their combination as a single ``if and only if'' theorem.
With multiple attractors the long-term dynamics of $f$ depends on the initial conditions
and can be a form of unpredictability if the boundaries of the basins of attraction are complicated.
The presence of many attractors indicates a high level of complexity and is important in diverse areas of application \cite{Fe08}.

Given a multi-parameter collection of maps of the form \eqref{eq:f},
a practical approach for finding such codimension-three scenarios
is to numerically solve the three codimension-one conditions \eqref{eq:threeConditions}.
Once this is achieved, one can check that all the conditions in Theorems \ref{th:infinity} and \ref{th:homoclinic} are satisfied.
Section \ref{sec:examples} provided three examples
with parameter values accurate to ten decimal places.
In each case we showed the $\cX^k \cY$-cycles for the smallest eight values of $k$
for which they are admissible and asymptotically stable.

For each of the examples in \S\ref{sec:examples} (and those in \cite{Si14}), we have $d \ne 0$ and ${\rm gcd}(d,n) = 1$ \eqref{eq:dConstraints}.
Yet Theorems \ref{th:infinity} and \ref{th:homoclinic} involve no such restriction on the value of $d$.
It remains to be determined if the codimension-three scenario is feasible when these constraints on $d$ are not satisfied.
In such cases the subsumed homoclinic connection would not connect the points of the $\cX$-cycle with a single closed loop.
It also remains to identify the codimension-three scenario in a system of differential equations.
For instance, it has recently been shown that multiple attractors can be created in
grazing-sliding bifurcations of piecewise-smooth systems of ODEs
(which can be viewed as special types of border-collision bifurcations).
However, currently it is not known how many attractors be created \cite{GlKo12,GlKo16}.

%{\footnotesize
%\bibliographystyle{unsrt}
%\bibliography{../DynSyst,../MathBio,../Misc,../OtherTheory,../PWS,../Stoch}

\begin{thebibliography}{10}

\bibitem{PaTa93}
J.~Palis and F.~Takens.
\newblock {\em Hyperbolicity and sensitive chaotic dynamics at homoclinic
  bifurcations.}
\newblock Cambridge University Press, New York, 1993.

\bibitem{GaSi72}
N.K. Gavrilov and L.P. \v{S}il'nikov.
\newblock On three-dimensional dynamical systems close to systems with a
  structurally unstable homoclinic curve {I}.
\newblock {\em Mat. USSR Sb.}, 17:467--485, 1972.

\bibitem{GaSi73}
N.K. Gavrilov and L.P. \v{S}il'nikov.
\newblock On three-dimensional dynamical systems close to systems with a
  structurally unstable homoclinic curve {II}.
\newblock {\em Mat. USSR Sb.}, 19:139--156, 1973.

\bibitem{GoBa02}
B.K. Goswami and S.~Basu.
\newblock Self-similar organization of {G}avrilov-{S}ilnikov-{N}ewhouse sinks.
\newblock {\em Phys. Rev. E}, 65:036210, 2002.

\bibitem{Ne74}
S.E. Newhouse.
\newblock Diffeomorphisms with infinitely many sinks.
\newblock {\em Topology}, 12:9--18, 1974.

\bibitem{Ro83}
C.~Robinson.
\newblock Bifurcation to infinitely many sinks.
\newblock {\em Commun. Math. Phys.}, 90:433--459, 1983.

\bibitem{GoSh96}
S.V. Gonchenko, L.P. Shil'nikov, and D.V. Turaev.
\newblock Dynamical phenomena in systems with structurally unstable
  {P}oincar\'{e} homoclinic orbits.
\newblock {\em Chaos}, 6(1):15--31, 1996.

\bibitem{GoSh97}
S.V. Gonchenko, L.P. Shil'nikov, and D.V. Turaev.
\newblock Quasiattractors and homoclinic tangencies.
\newblock {\em Computers Math. Applic.}, 34(2-4):195--227, 1997.

\bibitem{GoTu05}
S.V. Gonchenko, D.V. Turaev, and L.P. Shil'nikov.
\newblock On dynamic properties of diffeomorphisms with homoclinic tangency.
\newblock {\em J. Math. Sci.}, 126(4):1317--1343, 2005.
\newblock Translated from {\em Sovremennaya Matematika i Ee Prilozheniya},
  Vol.~7, Suzdal Conference-1, 2003.

\bibitem{Da91}
G.J. Davis.
\newblock Infinitely many coexisting sinks from degenerate homoclinic
  tangencies.
\newblock {\em Trans. Am. Math. Soc.}, 323(2):727--748, 1991.

\bibitem{GoTu07}
S.~Gonchenko, D.~Turaev, and L.~Shilnikov.
\newblock Homoclinic tangencies of arbitrarily high orders in conservative and
  dissipative two-dimensional maps.
\newblock {\em Nonlinearity}, 20:241--275, 2007.

\bibitem{GoSh05}
S.V. Gonchenko and L.P. Shilnikov.
\newblock On two-dimensional area-preserving maps with homoclinic tangencies
  that have infinitely many generic elliptic periodic points.
\newblock {\em J. Math. Sci.}, 128(2):2767--2773, 2005.

\bibitem{GoGo09}
M.S. Gonchenko and S.V. Gonchenko.
\newblock On cascades of elliptic periodic points in two-dimensional symplectic
  maps with homoclinic tangencies.
\newblock {\em Regul. Chaotic Dyns.}, 14(1):116--136, 2009.

\bibitem{HiLa95}
P.~Hirschberg and C.~Laing.
\newblock Successive homoclinic tangencies to a limit cycle.
\newblock {\em Phys. D}, 89:1--14, 1995.

\bibitem{ChRo99}
A.R. Champneys and A.J. Rodriguez-Luis.
\newblock The non-transverse {S}hil'nikov-hopf bifurcation: uncoupling of
  homoclinic orbits and homoclinic tangencies.
\newblock {\em Phys. D}, 128:130--158, 1999.

\bibitem{Si16b}
D.J.W. Simpson.
\newblock Unfolding homoclinic connections formed by corner intersections in
  piecewise-smooth maps.
\newblock {\em Chaos}, 26:073105, 2016.

\bibitem{DiBu08}
M.~di~Bernardo, C.J. Budd, A.R. Champneys, and P.~Kowalczyk.
\newblock {\em Piecewise-smooth Dynamical Systems. Theory and Applications.}
\newblock Springer-Verlag, New York, 2008.

\bibitem{Si16}
D.J.W. Simpson.
\newblock Border-collision bifurcations in $\mathbb{R}^n$.
\newblock {\em SIAM Rev.}, 58(2):177--226, 2016.

\bibitem{KaMa98}
T.~Kapitaniak and Yu. Maistrenko.
\newblock Multiple choice bifurcations as a source of unpredictability in
  dynamical systems.
\newblock {\em Phys. Rev. E}, 58(4):5161--5163, 1998.

\bibitem{DuNu99}
M.~Dutta, H.E. Nusse, E.~Ott, J.A. Yorke, and G.~Yuan.
\newblock Multiple attractor bifurcations: A source of unpredictability in
  piecewise smooth systems.
\newblock {\em Phys. Rev. Lett.}, 83(21):4281--4284, 1999.

\bibitem{ZhMo08d}
Z.T. Zhusubaliyev, E.~Mosekilde, and S.~Banerjee.
\newblock Multiple-attractor bifurcations and quasiperiodicity in
  piecewise-smooth maps.
\newblock {\em Int. J. Bifurcation Chaos}, 18(6):1775--1789, 2008.

\bibitem{Si14}
D.J.W. Simpson.
\newblock Sequences of periodic solutions and infinitely many coexisting
  attractors in the border-collision normal form.
\newblock {\em Int. J. Bifurcation Chaos}, 24(6):1430018, 2014.

\bibitem{Si14b}
D.J.W. Simpson.
\newblock Scaling laws for large numbers of coexisting attracting periodic
  solutions in the border-collision normal form.
\newblock {\em Int. J. Bifurcation Chaos}, 24(9):1450118, 2014.

\bibitem{DoLa08}
Y.~Do and Y.-C. Lai.
\newblock Multistability and arithmetically period-adding bifurcations in
  piecewise smooth dynamical systems.
\newblock {\em Chaos}, 18:043107, 2008.

\bibitem{GaTr83}
J.M. Gambaudo and C.~Tresser.
\newblock Simple models for bifurcations creating horseshoes.
\newblock {\em J. Stat. Phys.}, 32(3):455--476, 1983.

\bibitem{DoKi08}
Y.~Do, S.D. Kim, and P.S. Kim.
\newblock Stability of fixed points placed on the border in the piecewise
  linear systems.
\newblock {\em Chaos Solitons Fractals}, 38(2):391--399, 2008.

\bibitem{SiMe09}
D.J.W. Simpson and J.D. Meiss.
\newblock Shrinking point bifurcations of resonance tongues for
  piecewise-smooth, continuous maps.
\newblock {\em Nonlinearity}, 22(5):1123--1144, 2009.

\bibitem{Si15c}
D.J.W. Simpson.
\newblock The structure of mode-locking regions of piecewise-linear continuous
  maps.
\newblock {\em Unpublished}, 2015.

\bibitem{Di03}
M.~di~Bernardo.
\newblock Normal forms of border collision in high dimensional non-smooth maps.
\newblock In {\em Proceedings IEEE ISCAS, Bangkok, Thailand}, volume~3, pages
  76--79, 2003.

\bibitem{Fe78}
M.I. Feigin.
\newblock On the structure of {$C$}-bifurcation boundaries of
  piecewise-continuous systems.
\newblock {\em J. Appl. Math. Mech.}, 42(5):885--895, 1978.
\newblock Translation of {\em Prikl.~Mat.~Mekh.}, 42(5):820-829, 1978.

\bibitem{Si14d}
D.J.W. Simpson.
\newblock On the relative coexistence of fixed points and period-two solutions
  near border-collision bifurcations.
\newblock {\em Appl. Math. Lett.}, 38:162--167, 2014.

\bibitem{Gl14}
P.~Glendinning.
\newblock Invariant measures for the $n$-dimensional border collision normal
  form.
\newblock {\em Int. J. Bifurcation Chaos}, 24(12):1450164, 2014.

\bibitem{Gl15b}
P.~Glendinning.
\newblock Bifurcation from stable fixed point to ${N}$-dimensional attractor in
  the border collision normal form.
\newblock {\em Nonlinearity}, 28:3457--3464, 2015.

\bibitem{Fe08}
U.~Feudel.
\newblock Complex dynamics in multistable systems.
\newblock {\em Int. J. Bifurcation Chaos}, 18(6):1607--1626, 2008.

\bibitem{GlKo12}
P.~Glendinning, P.~Kowalczyk, and A.B. Nordmark.
\newblock Attractors near grazing-sliding bifurcations.
\newblock {\em Nonlinearity}, 25:1867--1885, 2012.

\bibitem{GlKo16}
P.~Glendinning, P.~Kowalczyk, and A.B. Nordmark.
\newblock Multiple attractors in grazing-sliding bifurcations in
  {F}ilippov-type flows.
\newblock To appear: {\em IMA J.~Appl.~Math.}, 2016.

\end{thebibliography}
%}

\end{document}